\documentclass[11pt]{article}
\usepackage[
  left=2.7cm,
  right=2.7cm,
  top=2.8cm,
  bottom=2.8cm
]{geometry}

\usepackage[T1]{fontenc}
\usepackage[utf8]{inputenc}

\usepackage{amssymb,mathtools}

\usepackage{microtype}

\usepackage{enumitem}
\usepackage{subcaption}

\usepackage{xcolor}
\usepackage{booktabs}

\usepackage{pgfplots}
\pgfplotsset{compat=1.18}

\usepackage[colorlinks=true,linkcolor=blue,citecolor=blue,urlcolor=blue]{hyperref}

\usepackage{dsfont}
\usepackage{mathrsfs}
\newcommand{\1}{\mathds{1}}

\usepackage{amsthm}
\newtheorem{theorem}{Theorem}[section]

\newtheorem{proposition}[theorem]{Proposition}
\newtheorem{lemma}[theorem]{Lemma}
\newtheorem{corollary}[theorem]{Corollary}

\theoremstyle{definition}
\newtheorem{definition}[theorem]{Definition}
\theoremstyle{remark}
\newtheorem{remark}[theorem]{Remark}

\theoremstyle{definition}
\newtheorem{example}{Example}[section]

\DeclareMathOperator{\argmin}{argmin}

\newcommand{\R}{\mathbb{R}}
\newcommand{\E}{\mathbb{E}}
\newcommand{\MoreauYosida}[2]{\operatorname{e}_{#2} #1}

\newcommand{\proj}{\operatorname{proj}}

\newcommand{\dom}{\operatorname{dom}}

\makeatletter
\def\namedlabel#1#2{\begingroup
    #2%
    \def\@currentlabel{#2}%
    \phantomsection\label{#1}\endgroup}

\usepackage{todonotes}

\title{Sharp bounds for stochastic proximal and projection estimators via radial dominance}

\author{%
Gonzalo Contador\thanks{Universidad Técnica Federico Santa María, 
           Santiago, Chile. \texttt{gonzalo.contador@usm.cl}}
\and 
  Pedro P\'erez-Aros\thanks{Departamento de Ingenier\'ia Matem\'atica and Center for Mathematical Modeling (CNRS IRL2807), Universidad de Chile, Santiago, Chile. \texttt{pperez@dim.uchile.cl}}
  \and
  Emilio Vilches\thanks{Instituto de Ciencias de la Ingenier\'ia, Universidad de O'Higgins, Rancagua, Chile; and Center for Mathematical Modeling (CNRS IRL2807), Universidad de Chile, Santiago, Chile. \texttt{emilio.vilches@uoh.cl}}%
}

\date{\today}

\begin{document}

\maketitle

\begin{abstract}
We study stochastic barycentric estimators for proximal points and metric
projections obtained by exponentially reweighting Gaussian perturbations.
Our main result is an abstract comparison theorem for probability measures
with densities proportional to an exponential weight, under a radial
dominance condition relative to a prescribed profile. This yields an
explicit upper bound for the norm of the associated barycenter in terms of
a one-dimensional comparison measure. We also provide tractable sufficient
conditions for radial dominance, including strong convexity, addition of
nonnegative convex terms, and star-shaped constraints.

As a consequence, we obtain a refined convergence rate for stochastic
proximal estimators of weakly convex functions, together with asymptotic
sharpness of the constant. The same framework yields a corresponding rate
for stochastic projection estimators onto closed convex sets. We further establish basic structural properties of the barycentric
approximation operator, such as smoothness and cocoercivity. Numerical
experiments illustrate the predicted rate, the dimensional
scaling of the constant, and its asymptotic sharpness.
\end{abstract}

\medskip

\section{Introduction}

Proximal mappings and metric projections are fundamental objects in
variational analysis and optimization. Introduced by Moreau in
\cite{MR201952}, proximal mappings now play a central role in modern
optimization through proximal algorithms, splitting methods, and their many
variants; see, for example, \cite{MR3616647}. Given a proper function
\(f\colon \R^n\to \R\cup\{+\infty\}\) and \(\lambda>0\), the proximal point
\(\operatorname{prox}_{\lambda f}(x)\) is the unique minimizer of
\(y\mapsto f(y)+\frac{1}{2\lambda}\|x-y\|^2\) whenever \(f\) is
\(\rho\)-weakly convex and \(0<\lambda<1/\rho\). In the particular case
\(f=\iota_C\), where \(C\subset \R^n\) is a closed convex set, the proximal
mapping reduces to the metric projection \(\proj_C\).

In many applications, however, the exact evaluation of 
\(\operatorname{prox}_{\lambda f}(x)\) or \(\proj_C(x)\) is computational demanding, and sometimes impractical. This motivates the study of stochastic or barycentric
approximations built from random perturbations and exponential
reweighting. In this work, for \(\delta>0\), we consider the estimator
\begin{equation}\label{eq:intro_mdelta}
m_\delta(x)
:=
\frac{\E_{Y\sim\mathcal{N}(x,\delta\lambda I)}
\!\big[Y\,e^{-f(Y)/\delta}\big]}
{\E_{Y\sim\mathcal{N}(x,\delta\lambda I)}
\!\big[e^{-f(Y)/\delta}\big]}.
\end{equation}
When \(f=\iota_C\), this reduces to the conditional mean
\[
p_\delta(x)=\E[Y\mid Y\in C],\qquad Y\sim \mathcal N(x,\delta I),
\]
which may be viewed as a stochastic projection estimator or smoothed
projection onto \(C\). The estimator \(m_\delta(x)\) is of interest for several reasons. First, it provides a zeroth-order approximation of the proximal mapping, since its evaluation only requires function values of \(f\) and does not rely on subgradient information or on the explicit computation of \(\operatorname{prox}_{\lambda f}(x)\). Second, because it is defined through Gaussian expectations, it is naturally suited to Monte Carlo approximation and parallel implementation. Third, in contrast with the generally nonsmooth proximal mapping, the map \(x\mapsto m_\delta(x)\) is smooth, which makes it attractive for sensitivity analysis and differentiable optimization schemes. Thus, the barycentric estimator offers a tractable and regular surrogate of the proximal point, together with explicit nonasymptotic error bounds.

Estimators of the form \eqref{eq:intro_mdelta} arise naturally from
Laplace-type asymptotics for integrals with kernel \(e^{-g/\delta}\), where
\(g(y):=f(y)+\frac{1}{2\lambda}\|x-y\|^2\). This perspective appears
prominently in recent work. Motivated by Hamilton-Jacobi equations,
Osher, Heaton, and Fung \cite{MR4581306} introduced the HJ-Prox method and
proved convergence of barycentric approximations toward proximal points
under mild assumptions. More recently, \cite{MR4923371} developed a
self-normalized Laplace approximation framework that also covers smoothed
projection estimators and establishes asymptotic consistency under weak
local hypotheses.

More broadly, estimator \eqref{eq:intro_mdelta} belongs to the
Gaussian-smoothing and zeroth-order lineage, in which Gaussian expectations
are used to build smooth surrogates of nonsmooth objectives and their
proximal maps \cite{NesterovSpokoiny2017}. A common feature of these works
is that they establish \emph{qualitative} convergence of the estimator to
the exact proximal point as the smoothing parameter \(\delta\) vanishes,
but stop short of identifying the exact rate at which this convergence
occurs or the sharp constant governing it. This gap is not merely academic:
when \eqref{eq:intro_mdelta} is evaluated by Monte Carlo sampling, the
variance of the estimator deteriorates as \(\delta\downarrow 0\)
\cite{NaldiLabarriereMolinariVilla2026}, so that the smallest admissible
\(\delta\) attaining a target accuracy (and hence the sampling cost) is
dictated precisely by the constant in the error bound. Our contribution is
therefore not the estimator itself, nor the fact that it converges, both of
which are known, but rather a \emph{sharp} nonasymptotic bound with explicit
dimensional dependence, together with a matching asymptotic-sharpness
certificate showing that the constant cannot be improved within this
estimator class.  

A related fixed-parameter perspective on the same barycentric
operator appears in \cite{NaldiLabarriereMolinariVilla2026}; the present
paper follows a different route, focusing on sharp one-step bias estimates
relative to the classical proximal and projection maps.

The main contribution of the present paper is the introduction of a new
technique, based on a radial dominance comparison principle, for the
analysis of barycentric estimators associated with densities proportional
to \(e^{-G/\delta}\). This approach yields explicit and, in a precise
sense, optimal nonasymptotic bounds for the approximation error, and in
particular refines the convergence estimates obtained in \cite{MPV2026}. More precisely, \cite{MPV2026} establishes the dimension-explicit bound
$\|m_\delta(x)-\operatorname{prox}_{\lambda f}(x)\|\le\sqrt{n\delta/\mu}$,
derived from a second-moment (variance) estimate followed by Jensen's
inequality. Here we sharpen the constant to
$\sqrt{2\delta/\mu}\,\Gamma\!\big(\tfrac{n+1}{2}\big)/\Gamma\!\big(\tfrac{n}{2}\big)$,
which is strictly smaller for every finite $n$ and asymptotically equivalent
to $\sqrt{n\delta/\mu}$ as $n\to\infty$. This sharp constant is exactly the
value attained, in the limit of vanishing aperture, by the extremal cone
example of \cite{MPV2026}; the radial dominance principle promotes it to a
uniform upper bound by controlling the first moment of the barycenter
directly, rather than passing through its second moment.
Under a radial dominance condition relative to a one-dimensional profile
\(\psi\), we compare the radial law of the barycenter with an explicit
one-dimensional reference measure, leading to sharp upper bounds for its
norm. We then apply this abstract principle to reduced proximal
objectives and derive improved convergence rates for stochastic proximal
and projection estimators.

\subsection*{Our contributions}

The main contributions of the paper may be summarized as follows.

\begin{itemize}
    \item We introduce a radial dominance framework for measures with
    density proportional to the negative exponential of a given function and prove an abstract
    stochastic comparison theorem that bounds the barycenter by the mean of
    an explicit one-dimensional comparison measure.

    \item We show that the resulting bound is asymptotically sharp and
    provide tractable sufficient conditions for radial dominance,
    including strong convexity, addition of nonnegative convex terms, and
    star-shaped constraints.

    \item We apply the abstract result to proximal estimators associated
    with \(\rho\)-weakly convex functions and obtain a refined rate of
    order \(\sqrt{\delta}\) with an explicit dimension-dependent constant.
    We also derive the corresponding projection result for conditional
    Gaussian barycenters over closed convex sets.

    \item We illustrate the theory on concrete examples, including
    Lasso-type objectives, and study basic structural properties of the
    barycentric approximation operator, such as smoothness, cocoercivity,
    and its use in inexact proximal-type iterations.
\item We complement the analysis with numerical experiments that
    confirm the $\sqrt{\delta}$ convergence rate, exhibit the dimensional
    dependence of the constant, and demonstrate its asymptotic sharpness
    through the extremal cone construction.
\end{itemize}
A key contribution of the paper is the identification of a simple geometric mechanism, namely radial dominance, which underlies both the abstract comparison theorem and its proximal applications in a unified framework.

The remainder of the paper is organized as follows.
Section~\ref{preliminaries} introduces the notation and preliminary material.
Section~\ref{radial-dominance-section} presents the radial dominance framework and its basic structural properties.
In Section~\ref{abstract-radial-dominance}, we establish the abstract radial dominance bound and derive a general proximal reduction theorem.
Section~\ref{weakly-convex} specializes these results to weakly convex proximal estimators and stochastic projection estimators, and discusses several illustrative examples.
Section~\ref{monotone-prox} is devoted to structural properties of the barycentric approximation operator and to a proximal-type iterative scheme based on these estimators. Section~\ref{numerics} reports numerical experiments that corroborate the theoretical rates and the sharpness of the constant.
We conclude with a discussion of final remarks and directions for future research.
\section{Mathematical Preliminaries}\label{preliminaries}
Throughout this paper, we work in the  Euclidean space $\mathbb{R}^n$ endowed with the inner product $\langle \cdot, \cdot\rangle$ and the associated norm $\Vert \cdot\Vert$. The closed unit ball in $\mathbb{R}^n$ is denoted by $\mathbb{B}$. Moreover, \(\mathbb S^{n-1}\) denotes the unit
sphere in \(\R^n\), and $\omega_{n-1}$ denotes its \((n-1)\)-dimensional surface measure, that is, $$\omega_{n-1}= \frac{2 \pi^{n/2}}{\Gamma\left(\frac{n}{2}\right)}.$$ We use the convention $\exp(+\infty) =+\infty $ and $\exp( -\infty ) = 0$.  For a set $A\subset\mathbb{R}^n$, its \emph{indicator function} is defined by $\iota_A(x)=0$ if $x\in A$, and $\iota_A(x)=+\infty$ otherwise. The \emph{distance function} from $x\in \mathbb{R}^n$  to a set $A$ is given by $d_{A}(x)=\inf\left\{  \|  x - y\| : y\in A\right\}$. If $A$ is nonempty, closed and convex, then $\operatorname{proj}_{A} (x)$ denotes the  \emph{metric projection} of $x\in \mathbb{R}^n$ onto $A$.  Given a function $g\colon \mathbb{R}^n \rightarrow \mathbb{R}\cup \{+\infty \}$, the \emph{domain} of $g$ is
\begin{equation*}
\operatorname{dom}(g):=\{x\in \mathbb{R}^n: g(x)<+\infty \}.
\end{equation*}
We say that $g$ is \emph{proper} if $\operatorname{dom}g\neq
\emptyset$. Given $\rho\geq 0$, a function $g\colon \R^n\to \mathbb{\R}\cup \{+\infty \}$ is called $\rho$-weakly convex if $g+\frac{\rho}{2}\|\cdot\|^2$ is convex. For $\mu>0$, we say that $g$ is $\mu$-strongly convex if $g-\frac{\mu}{2}\Vert \cdot\Vert^2$ is convex.

The Moreau envelope of index $\lambda>0$ of a function $g\colon \R^n\to \mathbb{R}\cup \{+\infty\}$ is   defined by 
\begin{equation*}
    \MoreauYosida{g}{\lambda}(x):=\inf_{y\in \mathbb{R}^n} \left(	g(y) + \frac{1}{2\lambda } \|  x - y\|^2		\right) \textrm{ for all } x\in \mathbb{R}^n.
\end{equation*}
 If $g$ is proper, lower semicontinuous, and $\rho$-weakly convex, and if $0<\lambda<\frac{1}{\rho}$ (with the convention $1/0=+\infty$), then the above infimum is attained at a unique point $\operatorname{prox}_{\lambda g}(x)\in \mathbb{R}^n$, and
\begin{equation*}
\begin{aligned}
\MoreauYosida{g}{\lambda}(x)&=g(\operatorname{prox}_{\lambda g}(x))+\frac{1}{2\lambda}\Vert x-\operatorname{prox}_{\lambda g}(x)\Vert^2.
\end{aligned}
\end{equation*}
In this case, the operator $x\mapsto \operatorname{prox}_{\lambda g}(x)$ is everywhere defined and is called the \emph{proximal operator} of $g$ of index $\lambda$. Moreover, $\MoreauYosida{g}{\lambda}$ is continuously differentiable and, for each  $x\in \mathbb{R}^n$, 
\begin{equation*}
\nabla \MoreauYosida{g}{\lambda}(x)=\frac{1}{\lambda}(x-\operatorname{prox}_{\lambda g}(x)).
\end{equation*}
We refer to \cite{MR3288271,MR3616647} for background on weakly convex functions, proximal mappings, and Moreau envelopes.

Let $\delta>0$, and let $g\colon \mathbb{R}^n \to \mathbb{R}\cup \{+\infty\}$ be a proper function such that $e^{-g/\delta}$ is integrable and
$$
\int_{\mathbb{R}^n}e^{-g(y)/\delta}dy>0.
$$
Under these assumptions, we denote by \(\sigma_\delta^g\) the
probability measure on \(\R^n\) defined by
\[
\sigma_\delta^g(A)
:=
\frac{\displaystyle\int_A e^{-g(y)/\delta}\,dy}
{\displaystyle\int_{\R^n} e^{-g(y)/\delta}\,dy}
\]
for every Lebesgue-measurable set \(A\subset\R^n\). Hence, for every integrable
function \(w\), we write
\[
\E_{\sigma_\delta^g}[w]
:=
\frac{\displaystyle\int_{\R^n} w(y)e^{-g(y)/\delta}\,dy}
{\displaystyle\int_{\R^n} e^{-g(y)/\delta}\,dy}.
\]
When no confusion is possible, we simply write \(\sigma_\delta\) instead
of \(\sigma_\delta^g\). If \(Y\sim \sigma_\delta^g\), we denote by
\(\operatorname{Cov}_{\sigma_\delta^g}(Y)\) its covariance matrix,
namely
\[
\operatorname{Cov}_{\sigma_\delta^g}(Y)
:=
\E_{\sigma_\delta^g}
\bigl[(Y-\E_{\sigma_\delta^g}[Y])
(Y-\E_{\sigma_\delta^g}[Y])^\top\bigr].
\]
For symmetric matrices \(A,B\in\R^{n\times n}\), we write $A\preceq B$ when \(B-A\) is positive semidefinite.

Let \(f\colon \mathbb{R}^n\to \mathbb{R}\cup\{+\infty\}\) be a \(\rho\)-weakly convex function with \(\rho\ge 0\). 
In this work, we consider the barycenter estimator
\[
m_{\delta}(x;f)
:=
\frac{\mathbb{E}_{Y\sim \mathcal{N}(x,\delta \lambda I)}
\!\bigl[Y\,e^{-f(Y)/\delta}\bigr]}
{\mathbb{E}_{Y\sim \mathcal{N}(x,\delta \lambda I)}
\!\bigl[e^{-f(Y)/\delta}\bigr]}.
\]
If \(f\) is \(\rho\)-weakly convex for some \(\rho\ge 0\) and
\(\operatorname{dom}f\) has nonempty interior, then \(m_\delta(x;f)\) is
well defined for every \(\delta>0\) and every \(0<\lambda<1/\rho\). When
the function \(f\) is clear from the context, we simply write
\(m_\delta(x)\).

We also use the usual stochastic order (see, e.g., \cite{ShakedShanthikumar2007}): for real-valued random variables
\(X\) and \(Y\), we write \(X\le_{\mathrm{st}}Y\) if
\[
\mathbb P(X>t)\le \mathbb P(Y>t)
\qquad \text{for all } t\in\mathbb R.
\]
We conclude this section by establishing a covariance bound for nonsmooth
strongly log-concave measures.
\begin{lemma}\label{lem:covariance-strongly-log-concave}
Let \(V\colon \mathbb R^n\to(-\infty,+\infty]\) be proper, lower semicontinuous,
and \(\bar\mu\)-strongly convex for some \(\bar\mu>0\). Let \(\nu\) be the
probability measure defined by $d\nu(y):=Z^{-1}e^{-V(y)}\,dy$,  where
\[
0<Z:=\int_{\mathbb R^n}e^{-V(y)}\,dy<+\infty.
\]
Then $Y\sim \nu$ has finite second moment, and $\operatorname{Cov}_{\nu}(Y)\preceq \frac{1}{\bar\mu}I$. Equivalently, \(\operatorname{Var}_{\nu}(\langle a,Y\rangle)\le \|a\|^{2}/\bar\mu\)
for every \(a\in\mathbb R^n\).
\end{lemma}
\begin{proof}
Write \(V=H+\tfrac{\bar\mu}{2}\|\cdot\|^{2}\), where
\(H:=V-\tfrac{\bar\mu}{2}\|\cdot\|^{2}\) is proper, lower semicontinuous and
convex. Being proper, lsc and convex, \(H\) admits an affine minorant: there
exist \(b\in\mathbb R^n\), \(\beta\in\mathbb R\) with
\begin{equation}\label{eq:affine-minorant}
H(z)\ \ge\ \langle b,z\rangle+\beta\qquad\text{for all }z\in\mathbb R^n .
\end{equation}
Consequently \(V(y)\ge\langle b,y\rangle+\beta+\tfrac{\bar\mu}{2}\|y\|^{2}\), so $\int_{\mathbb{R}^n}(1+\|y\|^{2})e^{-V(y)}\,dy<+\infty$. Hence \(Y\) has finite second moment
and \(\operatorname{Cov}_\nu(Y)\) is well defined. Moreover \(Z>0\) forces
\(\operatorname{dom}V=\operatorname{dom}H\) to have positive Lebesgue measure, so
the convex set \(\operatorname{dom}H\) is full-dimensional; in particular
\(\operatorname{int}(\operatorname{dom}H)\neq\emptyset\) and its topological
boundary \(\partial(\operatorname{dom}H)\) is Lebesgue-null.\\
\noindent \emph{Step 1 (smooth case).}
If \(W\in C^{2}(\mathbb R^n)\) satisfies \(D^{2}W\succeq\bar\mu I\) and
\(\int_{\mathbb{R}^n} e^{-W}<+\infty\), the Brascamp-Lieb variance inequality
\cite[Theorem~4.1]{MR450480} applied to \(\varphi_a(y):=\langle a,y\rangle\)
(for which \(\nabla\varphi_a\equiv a\)) gives, writing
\(d\mu_W:=(\int_{\mathbb R^n} e^{-W})^{-1}e^{-W}\,dy\),
\[
\operatorname{Var}_{\mu_W}(\varphi_a)
\ \le\
\int_{\mathbb R^n}\big\langle (D^{2}W)^{-1}a,\,a\big\rangle\,d\mu_W
\ \le\
\frac{\|a\|^{2}}{\bar\mu},
\]
where in the last step we have used that $D^{2}W\succeq\bar\mu I$ implies $(D^{2}W)^{-1}\preceq\bar\mu^{-1}I\).\\
\noindent \emph{Step 2 (regularization).}
Let \(\rho\) be a symmetric \(C^\infty\) mollifier supported in the unit ball
with \(\int\rho=1\), and set \(\rho_\varepsilon(\cdot)=\varepsilon^{-n}\rho(\cdot/\varepsilon)\).
For \(\tau>0\) let \(H_\tau\) be the Moreau envelope
\[
H_\tau(y):=\inf_{z\in\mathbb R^n}\Big\{H(z)+\tfrac{1}{2\tau}\|y-z\|^{2}\Big\},
\]
which is finite, convex, and \(C^{1}\) with \(\tau^{-1}\)-Lipschitz gradient.
Fix sequences \(\tau_k\downarrow0\) and \(\varepsilon_k\downarrow0\) with
\(\varepsilon_k\le\tau_k\le 1\), and put
\[
H_k:=H_{\tau_k}*\rho_{\varepsilon_k},\qquad
V_k:=H_k+\tfrac{\bar\mu}{2}\|\cdot\|^{2},\qquad
d\nu_k:=Z_k^{-1}e^{-V_k}\,dy,\ \ Z_k:=\!\int_{\mathbb R^n} e^{-V_k}.
\]
Each \(H_k\) is convex (convolution of a convex function with a nonnegative
kernel) and \(C^\infty\), so \(V_k\in C^\infty(\mathbb R^n)\) and
\(D^{2}V_k=D^{2}H_k+\bar\mu I\succeq\bar\mu I\) everywhere.\\
\noindent \emph{Step 3 (uniform integrable domination).}
Minimizing the affine-plus-quadratic minorant obtained from
\eqref{eq:affine-minorant} gives, for all \(y\),
\(H_\tau(y)\ge\langle b,y\rangle+\beta-\tfrac{\tau}{2}\|b\|^{2}\). Since \(\rho\)
is symmetric with \(\operatorname{supp}\rho_{\varepsilon_k}\subseteq
\{|u|\le\varepsilon_k\}\), convolution preserves this bound:
\[
H_k(y)=\int_{\mathbb R^n} H_{\tau_k}(y-u)\,\rho_{\varepsilon_k}(u)\,du
\ \ge\ \langle b,y\rangle+\beta-\tfrac{\tau_k}{2}\|b\|^{2}
\ \ge\ \langle b,y\rangle+\beta-\tfrac12\|b\|^{2}.
\]
Hence, uniformly in \(k\),
\[
0\le e^{-V_k(y)}\le e^{\frac12\|b\|^{2}-\beta}\,
e^{-\langle b,y\rangle-\frac{\bar\mu}{2}\|y\|^{2}}=:g(y) \,\textrm{ and }\, \int_{\mathbb R^n}(1+\|y\|^{2})\,g(y)\,dy<\infty .
\]
\noindent \emph{Step 4 (a.e.\ pointwise convergence).}
We show \(V_k\to V\) Lebesgue-a.e.\\
If \(y\in\operatorname{int}(\operatorname{dom}H)\), then \(\partial H(y)\neq\emptyset\)
and the minimal-norm subgradient \(\|\partial^{0}H(y)\|\) is finite; the Moreau
envelope satisfies \(H_{\tau_k}(y)\uparrow H(y)\) as \(\tau_k\downarrow0\)
 (see, e.g., \cite[Thm~1.25]{RockafellarWets1998}). Using
\(\|\nabla H_{\tau_k}(y)\|\le\|\partial^{0}H(y)\|\) and the
\(\tau_k^{-1}\)-Lipschitz continuity of \(\nabla H_{\tau_k}\),
\[
|H_k(y)-H_{\tau_k}(y)|
\le \varepsilon_k\Big(\|\nabla H_{\tau_k}(y)\|+\tfrac{\varepsilon_k}{\tau_k}\Big)
\le \varepsilon_k\big(\|\partial^{0}H(y)\|+1\big)\xrightarrow[k\to\infty]{}0 ,
\]
since \(\varepsilon_k/\tau_k\le1\). Therefore \(H_k(y)\to H(y)\), and
\(V_k(y)\to V(y)\).\\
If \(y\notin\overline{\operatorname{dom}H}\), set \(d:=\operatorname{dist}(y,\operatorname{dom}H)>0\).
For \(|u|\le\varepsilon_k<d/2\) every competitor \(z\in\operatorname{dom}H\) obeys
\(\|(y-u)-z\|\ge d/2\); combining this with \eqref{eq:affine-minorant} and
minimizing over the admissible radius yields, for all small \(\tau_k\),
\[
H_k(y)\ \ge\ \inf_{|u|\le\varepsilon_k}H_{\tau_k}(y-u)
 \ge\ -\|b\|\big(\|y\|+1\big)+\beta+\frac{d^{2}}{8\tau_k}-\frac{\|b\|\,d}{2}
 \xrightarrow[k\to\infty]{}\ +\infty ,
\]
so \(V_k(y)\to+\infty=V(y)\).\\
The remaining set \(\partial(\operatorname{dom}H)\) is Lebesgue-null, so
\(e^{-V_k}\to e^{-V}\) a.e.\\
\noindent \emph{Step 5 (limit of moments).}
By Steps 3, 4 and dominated convergence (with envelopes \(g\), \(\|y\|g\),
\(\|y\|^{2}g\in L^{1}\)),
\[
Z_k\to Z>0,\qquad
\int_{\mathbb{R}^n} y\,e^{-V_k}\,dy\to\int_{\mathbb{R}^n} y\,e^{-V}\,dy,\qquad
\int_{\mathbb{R}^n} yy^{\top}e^{-V_k}\,dy\to\int_{\mathbb{R}^n} yy^{\top}e^{-V}\,dy .
\]
Hence the means \(m_k\) and second-moment matrices \(S_k\) of \(\nu_k\) converge
to those of \(\nu\), so
\(\operatorname{Cov}_{\nu_k}(Y)=S_k-m_km_k^{\top}\to
\operatorname{Cov}_{\nu}(Y)\) entrywise. In particular, for every
\(a\in\mathbb R^n\),
\[
\operatorname{Var}_{\nu_k}(\langle a,Y\rangle)
\ \longrightarrow\
\operatorname{Var}_{\nu}(\langle a,Y\rangle).
\]
Finally, applying Step 1 to \(W=V_k\) gives
\(\operatorname{Var}_{\nu_k}(\langle a,Y\rangle)\le\|a\|^{2}/\bar\mu\) for every
\(k\). Passing to the limit, we obtain 
\[
a^{\top}\operatorname{Cov}_{\nu}(Y)\,a
=\operatorname{Var}_{\nu}(\langle a,Y\rangle)
\le\frac{\|a\|^{2}}{\bar\mu}
=a^{\top}\Big(\tfrac1{\bar\mu}I\Big)a
\qquad\text{for all }a\in\mathbb R^n ,
\]
which is exactly \(\operatorname{Cov}_{\nu}(Y)\preceq\bar\mu^{-1}I\).
\end{proof}

\section{Radial Dominance Framework}\label{radial-dominance-section}
In this section, we introduce the radial dominance condition, which is the basic structural assumption underlying our analysis. Roughly speaking, it requires that the growth of a given function along every ray emanating from the origin be bounded from below by a prescribed one-dimensional profile. This condition allows us to reduce the study of the associated barycenter to a comparison with an explicit radial measure on \([0,\infty)\). We also record two simple but useful sufficient conditions for radial dominance, namely strong convexity and stability under the addition of nonnegative convex terms or star-shaped constraints. These results will serve as the main tools in the proof of the abstract comparison theorem and its proximal applications.

\begin{definition}Let $G\colon \mathbb{R}^n \to \mathbb{R}\cup \{+\infty\}$ be proper and lower semicontinuous, with $G(0)=0$, and let  $\psi\colon [0,\infty)\to [0,\infty)$ be a function. 
    \begin{itemize}
        \item $\psi$ is called an \emph{admissible profile} if it is convex, nondecreasing, locally absolutely continuous, satisfies $\psi(0)=0$, and 
        \begin{equation}\label{Admissible}
        M_{\psi} := \int_{0}^{\infty} r^{n-1}e^{-\psi(r)/\delta}\,dr <\infty.
        \end{equation}
        
        \item Given an admissible profile $\psi$, the \emph{comparison radial measure} is the probability measure  $\nu_{\psi,\delta}$ on $[0,\infty)$ whose density with respect to
    Lebesgue measure is
        $$
        \frac{d\nu_{\psi,\delta}}{dr}(r)=\frac{r^{n-1}e^{-\psi(r)/\delta}}{M_{\psi}}.
        $$
        We fix once and for all a random variable $\widetilde R$ with law $\nu_{\psi,\delta}$; it implies that  $\mathbb{E}[\widetilde R]=\int_0^\infty r\,d\nu_{\psi,\delta}(r)$ is the one-dimensional comparison constant that appears in the bounds below.
        \item   We say that $G$ satisfies the \emph{radial $\psi$-dominance condition} if
        \begin{equation}\label{dominance}
        G(r_2\omega)-G(r_1 \omega)\geq \psi(r_2)-\psi(r_1) \quad \textrm{ for all } 0\leq r_1\leq r_2 \textrm{ and } \omega \in \mathbb{S}^{n-1}.
        \end{equation}
    \end{itemize}
\end{definition}

The next lemma follows directly from \eqref{dominance}.
\begin{lemma}\label{radial-monotonicity}
Suppose that $G$ satisfies \eqref{dominance} with profile $\psi$. Then, for every $\omega\in\mathbb{S}^{n-1}$, the function $r\longmapsto G(r\omega)-\psi(r)$ is nondecreasing on $[0,\infty)$.
\end{lemma}
The next proposition shows that strong convexity yields radial dominance with a quadratic profile.
\begin{proposition}\label{dominance-strongly-convex}
If \(G\) is \(\mu\)-strongly convex on \(\mathbb{R}^n\) and has minimizer \(0\), then \(G\) satisfies \eqref{dominance} with profile \(\psi(r)=\frac{\mu}{2}r^2\).
\end{proposition}

\begin{proof}
Define \(H(z):=G(z)-\frac{\mu}{2}\|z\|^2\). Then \(H\) is convex. Since \(0\) minimizes \(G\), strong convexity yields \(G(z)\ge G(0)+\frac{\mu}{2}\|z\|^2\) for all \(z\in\mathbb R^n\), and therefore \(0\) minimizes \(H\). It follows that, for every \(\omega\in\mathbb S^{n-1}\), the map \(r\mapsto H(r\omega)\) is convex and minimized at \(0\), hence nondecreasing on \([0,\infty)\). Thus, for all \(0\le r_1\le r_2\),
\[
G(r_2\omega)-G(r_1\omega)\ge \frac{\mu}{2}(r_2^2-r_1^2)=\psi(r_2)-\psi(r_1),
\]
which proves \eqref{dominance}.
\end{proof}

The following proposition shows that radial dominance is stable under the addition of a nonnegative convex term and under star-shaped constraints.
\begin{proposition}\label{stability}
Assume that \(G_1\) satisfies \eqref{dominance} with profile \(\psi\). Then the following hold:
\begin{enumerate}
\item[(i)] If \(G=G_1+G_2\), where \(G_2 \colon \mathbb{R}^n\to[0,+\infty]\) is proper,
    lower semicontinuous, convex, and \(G_2(0)=0\), then \(G\) satisfies
    \eqref{dominance} with profile \(\psi\).
\item[(ii)] If \(G=G_1+\iota_C\), where \(C\subset \mathbb{R}^n\) is closed and
    star-shaped at \(0\), then \(G\) satisfies \eqref{dominance} with profile \(\psi\).
\end{enumerate}
\end{proposition}
\begin{proof}
(i): Fix \(\omega\in\mathbb S^{n-1}\) and \(0\le r_1\le r_2\). Since \(G_2\) is convex and \(G_2(0)=0\), the function \(r\mapsto G_2(r\omega)\) is convex on \([0,\infty)\) and attains its minimum at \(r=0\). Hence it is nondecreasing on \([0,\infty)\), so $G_2(r_2\omega)-G_2(r_1\omega)\ge 0$. Therefore,
\begin{equation*}
G(r_2\omega)-G(r_1\omega)
=
[G_1(r_2\omega)-G_1(r_1\omega)]
+
[G_2(r_2\omega)-G_2(r_1\omega)] \ge
\psi(r_2)-\psi(r_1).
\end{equation*}
Thus \(G\) satisfies \eqref{dominance} with profile \(\psi\).   \\
(ii): Fix \(\omega\in\mathbb S^{n-1}\) and \(0\le r_1\le r_2\). Since \(C\) is star-shaped at \(0\), whenever \(r_2\omega\in C\) we also have \(r_1\omega=\frac{r_1}{r_2}(r_2\omega)\in C\); equivalently, the map \(r\mapsto\iota_C(r\omega)\) is nondecreasing on \([0,\infty)\), with values in \(\{0,+\infty\}\). Consequently \(\iota_C(r_2\omega)\ge\iota_C(r_1\omega)\), so that
$\iota_C(r_2\omega)-\iota_C(r_1\omega)\ge 0$ in \([0,+\infty]\) (with the convention that the difference is \(0\) when both terms equal \(+\infty\), which is precisely the case \(r_1\omega\notin C\) forcing \(r_2\omega\notin C\)).
Therefore,
\begin{equation*}
\begin{aligned}
G(r_2\omega)-G(r_1\omega)
&=
[G_1(r_2\omega)-G_1(r_1\omega)]
+
[\iota_C(r_2\omega)-\iota_C(r_1\omega)]\ge
\psi(r_2)-\psi(r_1).
\end{aligned}
\end{equation*}
Thus \(G\) satisfies \eqref{dominance} with profile \(\psi\).
\end{proof}

\section{Abstract Radial Dominance Bound}\label{abstract-radial-dominance}

In this section, we establish the main abstract comparison result of the paper. Under the radial dominance condition introduced above, we show that the barycenter associated with a density proportional to \(e^{-G/\delta}\) is controlled by the mean of an explicit one-dimensional comparison measure determined by the profile \(\psi\). The proof relies on a reduction to the radial variable and a stochastic comparison argument based on the monotonicity induced by \eqref{dominance}. This yields a sharp upper bound for the norm of the barycenter, together with an asymptotic sharpness statement. The result will serve as the basic tool for the proximal and projection applications developed in the following sections.

\begin{theorem}\label{Stochastic-dominance}
    Let $G\colon \mathbb{R}^n \to \mathbb{R}\cup \{+\infty\}$ be a proper lower semicontinuous function with unique minimizer at $0$, and assume that $G(0)=0$. Let $\psi$ be an admissible profile satisfying \eqref{Admissible}, and suppose $G$ satisfies the radial $\psi$-dominance condition\eqref{dominance}. For $\delta>0$, define the probability measure \(\mu_\delta\) on
\(\mathbb{R}^n\) by
$$
d\mu_{\delta}(z):=\frac{e^{-G(z)/\delta}}{\int_{\mathbb{R}^n}e^{-G(z)/\delta}\,dz}\,dz, \quad z\in \mathbb{R}^n.
$$
Then the following assertions hold:
\begin{enumerate}
    \item[(i)] The measure $\mu_{\delta}$ is well-defined, and $\mathbb{E}_{\mu_{\delta}}[\Vert Z\Vert]<\infty$.
    \item[(ii)] The mean $\mathfrak{m}_\delta(G):=\mathbb{E}_{\mu_\delta}[Z]$ satisfies
    \begin{equation*}
        \|\mathfrak{m}_\delta(G)\|
        \le
        \mathbb{E}[\widetilde R]
        =
        \frac{\displaystyle \int_0^\infty r^n e^{-\psi(r)/\delta}\,dr}
        {\displaystyle \int_0^\infty r^{n-1} e^{-\psi(r)/\delta}\,dr}.
    \end{equation*}
    \item[(iii)] For every \(\varepsilon>0\), there exists a proper lower
    semicontinuous function \(G_\varepsilon\) satisfying
    \eqref{dominance}  such that
    \[
    \|\mathfrak{m}_\delta(G_\varepsilon)\|
    >
    \mathbb{E}[\widetilde R]-\varepsilon.
    \]
\end{enumerate}
\end{theorem}
Before present the proof of Theorem~\ref{Stochastic-dominance}, we need some technical results. 

The first lemma is an immediate consequence of the polar-coordinate decomposition of Lebesgue measure.
\begin{lemma}\label{polar-decomposition}
Under the assumptions of Theorem \ref{Stochastic-dominance}, let \(R:=\|Z\|\) with \(Z\sim\mu_\delta\). Then \(R\) has density
\[
p_R(r)=\frac{r^{n-1}A(r)}{\displaystyle\int_0^\infty s^{n-1}A(s)\,ds},
\qquad
A(r):=\int_{\mathbb S^{n-1}} e^{-G(r\omega)/\delta}\,d\omega,
\qquad r\ge 0.
\]
\end{lemma}

\begin{lemma}\label{likelihood-ratio}
Assume that \eqref{dominance} holds. Then the function
\[
\Lambda(r):=\frac{A(r)}{\omega_{n-1}e^{-\psi(r)/\delta}}
=
\frac{1}{\omega_{n-1}}\int_{\mathbb{S}^{n-1}} e^{-H(r\omega)/\delta}\,d\omega,
\qquad r\ge 0,
\]
is nonincreasing.
\end{lemma}
\begin{proof}
By Lemma \ref{radial-monotonicity}, for each fixed \(\omega\in\mathbb{S}^{n-1}\), the map \(r\mapsto H(r\omega)\) is nondecreasing. Hence, since \(x\mapsto e^{-x/\delta}\) is nonincreasing, the map \(r\mapsto e^{-H(r\omega)/\delta}\) is nonincreasing. It follows that its average over \(\mathbb{S}^{n-1}\), namely \(r\mapsto \frac{1}{\omega_{n-1}}\int_{\mathbb{S}^{n-1}} e^{-H(r\omega)/\delta}\,d\omega\), is also nonincreasing.
\end{proof}
\begin{proposition}\label{stochastic-dominance-lemma}
Under the hypotheses of Theorem \ref{Stochastic-dominance}, let
\(Z\sim\mu_\delta\), set \(R:=\|Z\|\), and let \(\widetilde R\sim\nu_{\psi,\delta}\).
Then \(R\le_{\textrm{st}}\widetilde R\). In particular, \(\mathbb{E}[R]<\infty\) and
\[
\mathbb{E}_{\mu_\delta}[\|Z\|]
=
\mathbb{E}[R]
\le
\mathbb{E}[\widetilde R].
\]
\end{proposition}
\begin{proof}
By Lemma \ref{polar-decomposition}, the density of \(R\) is \(p_R\), while the density of \(\widetilde R\sim\nu_{\psi,\delta}\) is \(p_{\widetilde R}\), where
\[
p_R(r)=\frac{r^{n-1}A(r)}{\displaystyle\int_0^\infty s^{n-1}A(s)\,ds},
\quad
p_{\widetilde R}(r)=
\frac{r^{n-1}\omega_{n-1}e^{-\psi(r)/\delta}}
{\displaystyle\int_0^\infty s^{n-1}\omega_{n-1}e^{-\psi(s)/\delta}\,ds},
\qquad r\ge 0.
\]
Moreover, by Lemma \ref{likelihood-ratio}, the function $r\mapsto \Lambda(r):=\frac{A(r)}{\omega_{n-1}e^{-\psi(r)/\delta}}$ is nonincreasing on \([0,\infty)\).

Fix \(t\ge 0\). For every \(0\le r_1<t<r_2\), the monotonicity of \(\Lambda\) yields $\Lambda(r_1)\ge \Lambda(r_2)$,  that is,
\[
A(r_1)e^{-\psi(r_2)/\delta}\ge A(r_2)e^{-\psi(r_1)/\delta}.
\]
Multiplying by \((r_1r_2)^{n-1}\) and integrating over
\(\{(r_1,r_2):0\le r_1<t<r_2<\infty\}\), we obtain
\[
\int_0^t r^{n-1}e^{-\psi(r)/\delta}\,dr
\int_t^\infty r^{n-1}A(r)\,dr
\le
\int_0^t r^{n-1}A(r)\,dr
\int_t^\infty r^{n-1}e^{-\psi(r)/\delta}\,dr.
\]
Dividing by
\[
\left(\int_0^\infty r^{n-1}A(r)\,dr\right)
\left(\int_0^\infty r^{n-1}e^{-\psi(r)/\delta}\,dr\right),
\]
it follows that
\[
\mathbb{P}(\tilde{R} \leq t) \mathbb{P}(R>t)\le \mathbb{P}(\widetilde R>t) \mathbb{P}(R\leq t),
\]
which is equivalent to
\[
\mathbb{P}(R>t)\le \mathbb{P}(\widetilde R>t).
\]
Hence \(R\le_{\textrm{st}}\widetilde R\).  Moreover, since \(R\) and \(\widetilde R\) are nonnegative and integrable,
\[
\mathbb{E}[R]
=
\int_0^\infty \mathbb{P}(R>t)\,dt
\le
\int_0^\infty \mathbb{P}(\widetilde R>t)\,dt
=
\mathbb{E}[\widetilde R].
\]
Finally, \(\mathbb{E}[R]=\mathbb{E}_{\mu_\delta}[\|Z\|]\), which gives the conclusion.
\end{proof}

\begin{proof}
(i): By \eqref{dominance} with \(r_1=0\) and \(r_2=\|z\|\), and using
\(G(0)=0\) and \(\psi(0)=0\), we obtain \(G(z)\ge \psi(\|z\|)\). Hence
\(e^{-G(z)/\delta}\le e^{-\psi(\|z\|)/\delta}\). Integrating in polar
coordinates, we get
\[
\int_{\mathbb R^n} e^{-G(z)/\delta}\,dz
\le
\omega_{n-1}\int_0^\infty r^{n-1}e^{-\psi(r)/\delta}\,dr
=
\omega_{n-1}M_\psi
<\infty,
\]
so \(\mu_\delta\) is well defined. Likewise,
\[
\int_{\mathbb R^n}\|z\|e^{-G(z)/\delta}\,dz
\le
\omega_{n-1}\int_0^\infty r^n e^{-\psi(r)/\delta}\,dr.
\]
It remains to show that the right-hand side is finite. Since \(\psi\) is
admissible, \(\int_0^\infty r^{n-1}e^{-\psi(r)/\delta}\,dr<\infty\), so
\(\psi\not\equiv 0\). Choose \(r_0>0\) such that \(\psi(r_0)>0\). Since
\(\psi\) is convex and \(\psi(0)=0\), the map \(r\mapsto \psi(r)/r\) is
nondecreasing on \((0,\infty)\). Therefore,
\(\psi(r)\ge (\psi(r_0)/r_0)\,r\) for all \(r\ge r_0\), and hence
\[
\int_0^\infty r^n e^{-\psi(r)/\delta}\,dr<\infty.
\]
Consequently, \(\mathbb E_{\mu_\delta}[\|Z\|]<\infty\).\\
(ii): By Jensen's inequality, one has
\[
\|\mathfrak{m}_\delta(G)\| = \|\mathbb E_{\mu_\delta}[Z]\| \le \mathbb E_{\mu_\delta}[\|Z\|] = \mathbb E[R].
\]
Hence, by Proposition \ref{stochastic-dominance-lemma},
\(\mathbb E[R]\le \mathbb{E}[\widetilde R]\), and the conclusion follows.\\
(iii): Fix \(\alpha\in(0,\pi/2)\) and \(v\in\mathbb S^{n-1}\), and define
\[
C_\alpha:=\{z\in\mathbb R^n:\langle z,v\rangle\ge \Vert z\Vert \cos\alpha\},
\qquad
G_\alpha(z):=\psi(\|z\|)+\iota_{C_\alpha}(z).
\]
Then \(G_\alpha(0)=0\), and  Proposition \ref{stability}(ii) shows that \(G_\alpha\) satisfies \eqref{dominance} with profile \(\psi\). Moreover, since \(C_\alpha\) is a cone, one has \(G_\alpha(r\omega)=\psi(r)\) for every ray contained in \(C_\alpha\). In particular, equality holds in \eqref{dominance} along such rays. Furthermore, under \(\mu_\delta\), the radius \(R:=\|Z\|\) and the direction
\(U:=Z/\|Z\|\) are independent, and \(R\sim \nu_{\psi,\delta}\). By the
rotational symmetry of \(C_\alpha\cap \mathbb S^{n-1}\) about the axis \(v\),
the vector \(\mathbb E[U]\) is parallel to \(v\). Writing
\(\Theta:=\angle(U,v)\), we obtain $\mathbb E[U]=\mathbb E[\cos\Theta\mid \Theta\le \alpha]\,v$, and therefore
\[
\|\mathfrak m_\delta(G_\alpha)\|
=
\|\mathbb E[Z]\|
=
\|\mathbb E[RU]\|
=
\mathbb{E}[\widetilde R]\,
\mathbb E[\cos\Theta\mid \Theta\le \alpha].
\]
Since \(\cos\Theta\to 1\) uniformly on \(\{0\le \Theta\le \alpha\}\) as
\(\alpha\to 0^+\), it follows that $\mathbb E[\cos\Theta\mid \Theta\le \alpha]\to 1$, 
and hence $\|\mathfrak m_\delta(G_\alpha)\|\to \mathbb{E}[\widetilde R]$.
\end{proof}
The next theorem is an immediate consequence of Theorem \ref{Stochastic-dominance} applied to the reduced function \(G\).
\begin{theorem}\label{general-theorem}
Let \(f\colon \mathbb R^n\to(-\infty,+\infty]\) be proper and lower semicontinuous,
let \(\lambda>0\), let \(x\in\mathbb R^n\), and define
\[
F_x(y):=f(y)+\frac{1}{2\lambda}\|y-x\|^2.
\]
Assume that \(p:=\operatorname{prox}_{\lambda f}(x)\) is well defined, and set
\[
G(z):=F_x(p+z)-F_x(p), \qquad z\in\mathbb R^n.
\]
Let \(\psi\) be an admissible profile, and assume that \(G\) satisfies
\eqref{dominance} with profile \(\psi\). Then the following hold:
\begin{enumerate}
    \item[(i)]
    \[
    \|m_\delta(x)-p\|
    \le
    \mathbb{E}[\widetilde R]
    =
    \frac{\displaystyle\int_0^\infty r^n e^{-\psi(r)/\delta}\,dr}
         {\displaystyle\int_0^\infty r^{n-1} e^{-\psi(r)/\delta}\,dr}.
    \]

    \item[(ii)] The constant \(\mathbb{E}[\widetilde R]\) is optimal within the class
    of proper lower semicontinuous functions \(G\) satisfying \eqref{dominance}
    with profile \(\psi\).

  \item[(iii)] If \(\psi(r)=\frac{\mu}{p}r^p\) for some \(p>1\) and \(\mu>0\), then
\[
\mathbb{E}[\widetilde R]
=
\left(\frac{p\delta}{\mu}\right)^{1/p}
\frac{\Gamma\!\left(\frac{n+1}{p}\right)}
{\Gamma\!\left(\frac{n}{p}\right)}.
\]
\end{enumerate}
\end{theorem}
\begin{proof}
Since \(p=\operatorname{prox}_{\lambda f}(x)\), the point \(p\) minimizes \(F_x\).
Hence \(G\) is proper, lower semicontinuous, satisfies \(G(0)=0\), and has
unique minimizer at \(0\). Let \(Y\sim \mu_\delta\), so that \(m_\delta(x)=\mathbb E[Y]\), and set
\(Z:=Y-p\). Since $G(z)=F_x(p+z)-F_x(p)$, the density of \(Z\) is proportional to
\[
e^{-F_x(p+z)/\delta}=e^{-G(z)/\delta}e^{-F_x(p)/\delta}.
\]
Hence \(Z\) is distributed according to the probability measure associated with \(G\). Therefore,
\[
m_\delta(x)-p=\mathbb E[Y-p]=\mathbb E[Z]=\mathfrak m_\delta(G).
\]
Since \(G\) satisfies \eqref{dominance} with profile \(\psi\), all the
hypotheses of Theorem \ref{Stochastic-dominance} are satisfied. Applying
Theorem \ref{Stochastic-dominance}(ii) to \(G\), we obtain
\[
\|m_\delta(x)-p\|
=
\|\mathfrak m_\delta(G)\|
\le
\mathbb{E}[\widetilde R],
\]
which proves (i). Assertion (ii) follows from Theorem \ref{Stochastic-dominance}(iii). Finally, (iii) follows from the change of variables \(u=\mu r^p/(p\delta)\), which gives
\[
\mathbb{E}[\widetilde R]
=
\frac{\displaystyle\int_0^\infty r^n e^{-\mu r^p/(p\delta)}\,dr}
     {\displaystyle\int_0^\infty r^{n-1} e^{-\mu r^p/(p\delta)}\,dr}
=
\left(\frac{p\delta}{\mu}\right)^{1/p}
\frac{\Gamma\!\left(\frac{n+1}{p}\right)}
     {\Gamma\!\left(\frac{n}{p}\right)}.
\]
\end{proof}
The following example specializes the abstract reduction of Theorem~\ref{general-theorem} to a prototypical composite objective arising in sparse regression. It shows that the quadratic proximal penalty alone already renders the reduced energy strongly convex, so that the radial dominance condition holds with a quadratic profile and the general bound collapses to the explicit $\sqrt{\delta}$ estimate that will be established in full generality in Section~\ref{weakly-convex}.
\begin{example}
Let \(f(y):=\frac12\|Ay-b\|^2+\tau\|y\|_1\), \(y\in\mathbb R^n\), where
\(A\in\mathbb R^{m\times n}\), \(b\in\mathbb R^m\), and \(\tau>0\). Fix
\(\lambda>0\) and \(x\in\mathbb R^n\), and define
\(F_x(y):=f(y)+\frac{1}{2\lambda}\|y-x\|^2\). Since
\(y\mapsto \frac12\|Ay-b\|^2\) is convex, \(y\mapsto \tau\|y\|_1\) is convex,
and \(y\mapsto \frac{1}{2\lambda}\|y-x\|^2\) is \(\frac1\lambda\)-strongly
convex, it follows that \(F_x\) is \(\mu\)-strongly convex with
\(\mu=\frac1\lambda\). Hence, as a consequence of Proposition~\ref{dominance-strongly-convex}, the function $F_x$ satisfies the radial $\psi$-dominance condition with  $\psi(r) = \frac{1}{2\lambda } r^2$. Therefore, by  
Theorem~\ref{general-theorem} we get that
\[
\|m_\delta(x)-\operatorname{prox}_{\lambda f}(x)\|
\le
\sqrt{2\lambda\delta}\,
\frac{\Gamma\!\left(\frac{n+1}{2}\right)}
{\Gamma\!\left(\frac{n}{2}\right)}.
\]
\end{example}
The preceding example highlights the regularizing effect of the strongly convex quadratic penalty. To further build intuition before deriving our sharp convergence rates, we next illustrate that for a fundamental class of non-smooth convex objectives, the barycentric estimator actually admits a closed-form structural representation.
\begin{example}
Consider a polyhedral convex function of the form
\[
f(y):=\max_{1\le i\le m}(a_i^\top y+b_i),
\qquad a_i\in\R^n,\quad b_i\in\R.
\]
For each \(i\), define the active region $C_i:=\{y\in\R^n:\ a_i^\top y+b_i=f(y)\}$. Each \(C_i\) is a closed convex polyhedron, and the family \((C_i)_{i=1}^m\)
covers \(\R^n\). After removing duplicate affine pieces, the interiors of
the sets \(C_i\) are pairwise disjoint. Choosing a measurable partition
\((D_i)_{i=1}^m\) of \(\R^n\) such that \(D_i\subset C_i\) for every \(i\),
we may write the barycenter estimator cell by cell. Indeed, on each \(D_i\), one has
\[
f(y)+\frac{1}{2\lambda}\|y-x\|^2
=
a_i^\top y+b_i+\frac{1}{2\lambda}\|y-x\|^2.
\]
Completing the square gives
\[
a_i^\top y+b_i+\frac{1}{2\lambda}\|y-x\|^2
=
\frac{1}{2\lambda}\|y-(x-\lambda a_i)\|^2
+
\left(
a_i^\top x+b_i-\frac{\lambda}{2}\|a_i\|^2
\right).
\]
Set
\[
x_i:=x-\lambda a_i,
\qquad
w_i(x):=
\exp\!\left(
-\frac{1}{\delta}
\left(
a_i^\top x+b_i-\frac{\lambda}{2}\|a_i\|^2
\right)
\right),
\]
and let \(\gamma_{i,\delta}\) denote the Gaussian probability measure
\(\mathcal N(x_i,\delta\lambda I)\). Then
\[
m_\delta(x;f)
=
\frac{\displaystyle\sum_{i=1}^m
w_i(x)\int_{D_i} y\,d\gamma_{i,\delta}(y)}
{\displaystyle\sum_{i=1}^m
w_i(x)\gamma_{i,\delta}(D_i)}=\sum_{i=1}^m \pi_i(x)\,\mu_i(x),
\]
where
\[
\mu_i(x):=
\frac{1}{\gamma_{i,\delta}(D_i)}
\int_{D_i} y\,d\gamma_{i,\delta}(y) \textrm{ and } \pi_i(x):=
\frac{w_i(x)\gamma_{i,\delta}(D_i)}
{\sum_{j=1}^m w_j(x)\gamma_{j,\delta}(D_j)}.
\]
Here $\mu_i(x)$ is the mean of the Gaussian measure \(\gamma_{i,\delta}\) truncated to \(D_i\). Thus, for any polyhedral convex function, the stochastic proximal estimator
is a convex combination of truncated Gaussian means associated with the
active polyhedral regions of \(f\).
\end{example}

\section{Sharp Convergence Rate for the Proximal Estimator}\label{weakly-convex}

In this section, we specialize the abstract radial dominance principle developed above to the proximal setting of weakly convex functions. More precisely, we derive an explicit convergence rate for the stochastic barycentric estimator \(m_\delta(x)\) toward the exact proximal point \(\operatorname{prox}_{\lambda f}(x)\). The key observation is that, after recentering the regularized objective at the proximal point, the resulting reduced energy is strongly convex and hence satisfies the radial dominance condition with a quadratic profile. In this way, we obtain a sharp bound of order \(\sqrt{\delta}\), which improves the estimate from \cite{MPV2026}, and we also recover, as a direct consequence, the corresponding result for stochastic projection estimators onto closed convex sets.
\begin{theorem}\label{convergence}
Let $f\colon \mathbb{R}^n \to \mathbb{R}\cup \{+\infty\}$ be proper,
lower semicontinuous, and $\rho$-weakly convex for some $\rho\geq 0$.
Assume that $\operatorname{dom}f$ has nonempty interior. Fix
$x\in \mathbb{R}^n$ and $0<\lambda<1/\rho$, and set $\mu:=\frac{1}{\lambda}-\rho>0$.
Then, for every $\delta>0$,
\[
\Vert m_{\delta}(x)-\operatorname{prox}_{\lambda f}(x)\Vert
\leq
\sqrt{\frac{2\delta}{\mu}}\,
\frac{\Gamma\left(\frac{n+1}{2}\right)}
{\Gamma\left(\frac{n}{2}\right)}.
\]
Moreover, this estimate is asymptotically sharp: for every $x\in \mathbb{R}^n$ and every 
\(\varepsilon>0\), there exists a proper lower semicontinuous
\(\rho\)-weakly convex function \(f\) such that
\[
\sqrt{\frac{2\delta}{\mu}}\,
\frac{\Gamma\!\left(\frac{n+1}{2}\right)}
{\Gamma\!\left(\frac{n}{2}\right)}
-\varepsilon
<
\|m_{\delta}(x)-\operatorname{prox}_{\lambda f}(x)\|.
\]
\end{theorem}
\begin{proof}
Set $F_x(y):=f(y)+\frac{1}{2\lambda}\|y-x\|^2$ and $p:=\operatorname{prox}_{\lambda f}(x)$. Since \(f\) is \(\rho\)-weakly convex and \(\mu=\frac1\lambda-\rho>0\),
the function \(F_x\) is \(\mu\)-strongly convex. Hence the reduced function $G(z):=F_x(p+z)-F_x(p)$ is proper, lower semicontinuous, and \(\mu\)-strongly convex with minimizer
\(0\). Therefore, by Proposition~\ref{dominance-strongly-convex},
\(G\) satisfies \eqref{dominance} with profile
\(\psi(r)=\frac{\mu}{2}r^2\). Applying Theorem~\ref{general-theorem},
we obtain
\[
\|m_{\delta}(x)-p\|
\le
\mathbb{E}[\widetilde R]
=
\sqrt{\frac{2\delta}{\mu}}\,
\frac{\Gamma\!\left(\frac{n+1}{2}\right)}
{\Gamma\!\left(\frac{n}{2}\right)}.
\]
For asymptotic sharpness, fix \(x\in\mathbb R^n\), a unit vector \(v\in\mathbb S^{n-1}\), and \(\alpha\in(0,\pi/2)\), and recall the closed convex cone $C_\alpha:=\{z\in\mathbb R^n:\langle z,v\rangle\ge \|z\|\cos\alpha\}$ already used in the sharpness construction of Theorem~\ref{Stochastic-dominance}(iii). Set
\[
f_{x,\alpha}(y):=-\frac{\rho}{2}\|y-x\|^2+\iota_{x+C_\alpha}(y).
\]
Then \(f_{x,\alpha}\) is proper, lower semicontinuous, and \(\rho\)-weakly convex, and
\[
F_x(y)
=
f_{x,\alpha}(y)+\frac{1}{2\lambda}\|y-x\|^2
=
\frac{\mu}{2}\|y-x\|^2+\iota_{x+C_\alpha}(y),
\]
so \(\operatorname{prox}_{\lambda f_{x,\alpha}}(x)=x\), and  $G_\alpha(z)=\frac{\mu}{2}\|z\|^2+\iota_{C_\alpha}(z)$. By the sharpness construction in Theorem~\ref{Stochastic-dominance},
\[
\|m_{\delta,\alpha}(x)-x\|
\longrightarrow
\mathbb{E}[\widetilde R]
=
\sqrt{\frac{2\delta}{\mu}}\,
\frac{\Gamma\!\left(\frac{n+1}{2}\right)}
{\Gamma\!\left(\frac{n}{2}\right)}
\qquad\text{as }\alpha\downarrow 0.
\]
This proves the result.
\end{proof}
The next corollary specializes Theorem \ref{convergence} to the case of metric projections. If \(f=\iota_C\) for a nonempty closed convex set \(C\subset\mathbb R^n\), then \(\operatorname{prox}_{\lambda f}=\operatorname{proj}_C\), and the estimator \(m_\delta(x)\) reduces to the conditional expectation of a Gaussian random vector given that it belongs to \(C\). In this way, Theorem \ref{convergence} yields a sharp rate for the stochastic projection estimator.
\begin{corollary}
Let \(C\subset \mathbb{R}^n\) be a nonempty closed convex set with nonempty
interior. For \(x\in\mathbb{R}^n\) and \(\delta>0\), define $p_\delta(x):=\mathbb{E}[Y\mid Y\in C]$ and $Y\sim\mathcal N(x,\delta I)$. Then
\[
\|p_\delta(x)-\operatorname{proj}_C(x)\|
\le
\sqrt{2\delta}\,
\frac{\Gamma\!\left(\frac{n+1}{2}\right)}
{\Gamma\!\left(\frac{n}{2}\right)}.
\]
Moreover, this estimate is asymptotically sharp: for every
\(x\in\mathbb{R}^n\) and every \(\varepsilon>0\), there exists a nonempty
closed convex set \(C\subset\mathbb{R}^n\) with nonempty interior such that
\[
\sqrt{2\delta}\,
\frac{\Gamma\!\left(\frac{n+1}{2}\right)}
{\Gamma\!\left(\frac{n}{2}\right)}
-\varepsilon
<
\|p_\delta(x)-\operatorname{proj}_C(x)\|.
\]
\end{corollary}

\begin{proof}
Apply Theorem \ref{convergence} with \(f=\iota_C\), \(\rho=0\), and
\(\lambda=1\). Then \(\operatorname{prox}_{\lambda f}(x)=\operatorname{proj}_C(x)\),
and the barycenter estimator \(m_\delta(x)\) coincides with
\[
\frac{\mathbb E_{Y\sim\mathcal N(x,\delta I)}[Y\,\1_C(Y)]}
{\mathbb P(Y\in C)}
=
\mathbb E[Y\mid Y\in C]
=
p_\delta(x).
\]
Hence
\[
\|p_\delta(x)-\operatorname{proj}_C(x)\|
\le
\sqrt{2\delta}\,
\frac{\Gamma\!\left(\frac{n+1}{2}\right)}
{\Gamma\!\left(\frac{n}{2}\right)}.
\]
For asymptotic sharpness, fix \(x\in\mathbb{R}^n\), \(v\in\mathbb S^{n-1}\),
and \(\alpha\in(0,\pi/2)\), and consider the closed convex cone
\[
K_\alpha:=\{z\in\mathbb R^n:\langle z,v\rangle\ge \|z\|\cos\alpha\},
\qquad
C_\alpha:=x+K_\alpha.
\]
Then \(\operatorname{proj}_{C_\alpha}(x)=x\), and the corresponding estimator
\(p_{\delta,\alpha}(x)\) satisfies
\[
\|p_{\delta,\alpha}(x)-x\|
\longrightarrow
\sqrt{2\delta}\,
\frac{\Gamma\!\left(\frac{n+1}{2}\right)}
{\Gamma\!\left(\frac{n}{2}\right)}
\qquad\text{as }\alpha\downarrow 0
\]
by the sharpness construction in Theorem \ref{convergence}. This proves the result.
\end{proof}
\begin{remark}[Dimensional dependence]
While the bound in Theorem \ref{convergence} is exact, its dependence on the ambient dimension $n$ can be made more transparent by recalling the asymptotic expansion of the Gamma function. As $n \to \infty$, we have
\[
\frac{\Gamma\left(\frac{n+1}{2}\right)}{\Gamma\left(\frac{n}{2}\right)} \sim \sqrt{\frac{n}{2}}.
\]
Consequently, in high-dimensional settings, the approximation error scales precisely as $\mathcal{O}(\sqrt{n\delta/\mu})$, so that the leading-order behavior of our sharp constant recovers the bound $\sqrt{n\delta/\mu}$ of \cite{MPV2026}. This explicit $\sqrt{n}$ factor highlights the natural degradation of the stochastic barycentric estimator in high dimensions (the curse of dimensionality) and clarifies the exact trade-off required between the smoothing parameter $\delta$ and the dimension $n$ to maintain a target accuracy. We emphasize that this $\sqrt{n}$ growth is not an artifact of the analysis: by the asymptotic sharpness statement of Theorem~\ref{convergence}, the constant (and hence the $\sqrt{n}$ factor) is attained in the limit by the extremal cone configuration. Consequently, no refinement of the estimates can remove this dependence, and the barycentric estimator is \emph{provably} limited to accuracy $\mathcal{O}(\sqrt{n\delta/\mu})$ in the worst case over admissible objectives.

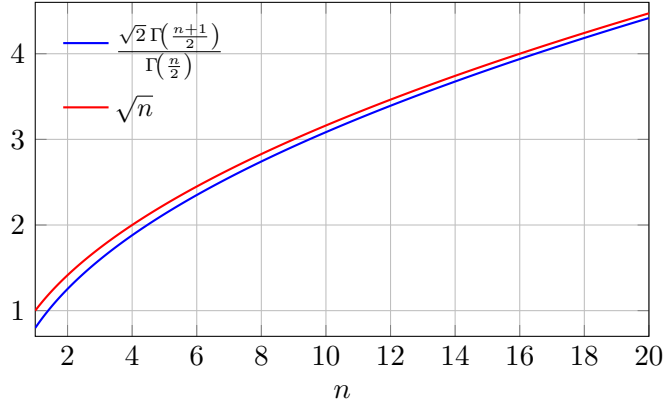
\begin{figure}[htbp]
\centering
\begin{tikzpicture}
\begin{axis}[
    width=9.7cm,
    height=6cm,
    xlabel={$n$},
    xmin=1, xmax=20,
    ymin=0.7, ymax=4.6,
    grid=both,
    legend style={
        at={(0.03,0.97)},
        anchor=north west,
        draw=none,
        fill=none,
        row sep=8pt,
        font=\small,
    },
    legend cell align={left},
]

\addplot[
    blue,
    thick,
    smooth,
    mark=none,
]
coordinates {
(1.0,0.79788456)
(1.1,0.85187938)
(1.2,0.90345163)
(1.3,0.95286237)
(1.4,1.00033231)
(1.5,1.04604962)
(1.6,1.09017595)
(1.7,1.13285113)
(1.8,1.17419691)
(1.9,1.21431995)
(2.0,1.25331414)
(2.1,1.29126262)
(2.2,1.32823935)
(2.3,1.36431036)
(2.4,1.39953492)
(2.5,1.43396639)
(2.6,1.46765300)
(2.7,1.50063849)
(2.8,1.53296264)
(2.9,1.56466177)
(3.0,1.59576912)
(3.1,1.62631517)
(3.2,1.65632798)
(3.3,1.68583342)
(3.4,1.71485539)
(3.5,1.74341603)
(3.6,1.77153591)
(3.7,1.79923414)
(3.8,1.82652853)
(3.9,1.85343571)
(4.0,1.87997121)
(4.1,1.90614959)
(4.2,1.93198450)
(4.3,1.95748878)
(4.4,1.98267447)
(4.5,2.00755295)
(4.6,2.03213492)
(4.7,2.05643052)
(4.8,2.08044930)
(4.9,2.10420032)
(5.0,2.12769216)
(5.1,2.15093297)
(5.2,2.17393048)
(5.3,2.19669203)
(5.4,2.21922462)
(5.5,2.24153490)
(5.6,2.26362922)
(5.7,2.28551364)
(5.8,2.30719394)
(5.9,2.32867563)
(6.0,2.34996401)
(6.1,2.37106412)
(6.2,2.39198082)
(6.3,2.41271873)
(6.4,2.43328231)
(6.5,2.45367583)
(6.6,2.47390339)
(6.7,2.49396893)
(6.8,2.51387623)
(6.9,2.53362895)
(7.0,2.55323059)
(7.1,2.57268454)
(7.2,2.59199403)
(7.3,2.61116223)
(7.4,2.63019214)
(7.5,2.64908670)
(7.6,2.66784873)
(7.7,2.68648095)
(7.8,2.70498599)
(7.9,2.72336642)
(8.0,2.74162468)
(8.1,2.75976316)
(8.2,2.77778417)
(8.3,2.79568995)
(8.4,2.81348267)
(8.5,2.83116442)
(8.6,2.84873723)
(8.7,2.86620309)
(8.8,2.88356391)
(8.9,2.90082155)
(9.0,2.91797782)
(9.1,2.93503447)
(9.2,2.95199320)
(9.3,2.96885568)
(9.4,2.98562351)
(9.5,3.00229826)
(9.6,3.01888145)
(9.7,3.03537458)
(9.8,3.05177907)
(9.9,3.06809634)
(10.0,3.08432776)
(10.1,3.10047466)
(10.2,3.11653834)
(10.3,3.13252007)
(10.4,3.14842108)
(10.5,3.16424258)
(10.6,3.17998575)
(10.7,3.19565172)
(10.8,3.21124163)
(10.9,3.22675656)
(11.0,3.24219758)
(11.1,3.25756573)
(11.2,3.27286203)
(11.3,3.28808747)
(11.4,3.30324303)
(11.5,3.31832966)
(11.6,3.33334827)
(11.7,3.34829979)
(11.8,3.36318510)
(11.9,3.37800506)
(12.0,3.39276054)
(12.1,3.40745235)
(12.2,3.42208131)
(12.3,3.43664823)
(12.4,3.45115388)
(12.5,3.46559902)
(12.6,3.47998440)
(12.7,3.49431076)
(12.8,3.50857882)
(12.9,3.52278927)
(13.0,3.53694281)
(13.1,3.55104012)
(13.2,3.56508186)
(13.3,3.57906867)
(13.4,3.59300119)
(13.5,3.60688006)
(13.6,3.62070588)
(13.7,3.63447926)
(13.8,3.64820078)
(13.9,3.66187103)
(14.0,3.67549058)
(14.1,3.68905998)
(14.2,3.70257978)
(14.3,3.71605053)
(14.4,3.72947274)
(14.5,3.74284694)
(14.6,3.75617364)
(14.7,3.76945334)
(14.8,3.78268654)
(14.9,3.79587371)
(15.0,3.80901534)
(15.1,3.82211189)
(15.2,3.83516381)
(15.3,3.84817157)
(15.4,3.86113561)
(15.5,3.87405636)
(15.6,3.88693426)
(15.7,3.89976972)
(15.8,3.91256316)
(15.9,3.92531499)
(16.0,3.93802562)
(16.1,3.95069544)
(16.2,3.96332484)
(16.3,3.97591420)
(16.4,3.98846390)
(16.5,4.00097432)
(16.6,4.01344581)
(16.7,4.02587874)
(16.8,4.03827347)
(16.9,4.05063034)
(17.0,4.06294970)
(17.1,4.07523188)
(17.2,4.08747722)
(17.3,4.09968606)
(17.4,4.11185870)
(17.5,4.12399548)
(17.6,4.13609671)
(17.7,4.14816269)
(17.8,4.16019374)
(17.9,4.17219015)
(18.0,4.18415222)
(18.1,4.19608025)
(18.2,4.20797452)
(18.3,4.21983532)
(18.4,4.23166292)
(18.5,4.24345761)
(18.6,4.25521965)
(18.7,4.26694933)
(18.8,4.27864689)
(18.9,4.29031261)
(19.0,4.30194674)
(19.1,4.31354953)
(19.2,4.32512125)
(19.3,4.33666213)
(19.4,4.34817242)
(19.5,4.35965237)
(19.6,4.37110220)
(19.7,4.38252217)
(19.8,4.39391249)
(19.9,4.40527340)
(20.0,4.41660512)
};
\addlegendentry{$\frac{\sqrt{2}\,\Gamma\!\left(\frac{n+1}{2}\right)}{\Gamma\!\left(\frac{n}{2}\right)}$}

\addplot[
    red,
    thick,
    domain=1:20,
    samples=300,
]
{sqrt(x)};
\addlegendentry{$\sqrt{n}$}

\end{axis}
\end{tikzpicture}
\caption{Comparison between the function
$\sqrt{2}\,\Gamma\!\left(\frac{n+1}{2}\right)/\Gamma\!\left(\frac{n}{2}\right)$
and the reference growth $\sqrt{n}$ as a function of the dimension $n$.}
\label{fig:gamma-ratio-vs-sqrtn}
\end{figure}
\end{remark}
\begin{remark}[Worst-case nature of the rate]
The bound of Theorem~\ref{convergence} holds uniformly over all base
points \(x\in\mathbb{R}^n\) and all admissible \(\rho\)-weakly convex
objectives, and should be read as a worst-case statement over
configurations rather than a pointwise description of the decay. Indeed,
at a base point \(x\) where the proximal map \(\operatorname{prox}_{\lambda
f}\) is single-valued and locally smooth, the reduced energy is locally
quadratic and the estimator error decays strictly faster than
\(\sqrt{\delta}\); the sharp \(\sqrt{\delta}\) rate, together with the
constant of Theorem~\ref{convergence}, is realized precisely at nonsmooth
configurations, such as vertices of polyhedral constraint sets or kinks of
the objective. This dichotomy is confirmed numerically in
Section~\ref{numerics}: at generic base points the measured ratio between
the error and the bound drifts toward zero as \(\delta\downarrow 0\),
whereas at the extremal cone and at nonsmooth kinks it saturates near the
predicted sharp constant.
\end{remark}

\section{Properties of a Barycentric Approximation Operator}\label{monotone-prox}
In this section, we discuss structural properties and algorithmic consequences of the barycentric approximation operator. Although the convergence results below can be derived from the general theory of inexact proximal point and inexact forward-backward methods (see, e.g., \cite{MR410483,MR1871872} and the recent survey \cite{LS2026}), we include them here for completeness and because, in the present setting, they admit a particularly transparent interpretation. In particular, the proofs make explicit how the barycentric approximation error enters the iteration and how the bounds established in the previous sections translate into convergence properties of the resulting schemes. We also emphasize that, in view of \cite{MR3942891}, the result of Subsection~\ref{monotonicity-sec} provides a natural basis for the use of Krasnosel'ski\u{\i}-Mann type iterations.

Throughout this section we make the dependence on the underlying function explicit and write $m_\delta(x;f)$ for the barycentric estimator introduced in \eqref{eq:intro_mdelta}. This notation is convenient here because we regard $m_\delta(\cdot;f)$ as an operator on $\mathbb{R}^n$ for a fixed $f$, study its regularity as $x$ varies, and, in the iterative scheme below, evaluate it at shifted arguments of the form $x-\lambda\nabla h(x)$.

The results in this section concern the vanishing-parameter or summable-error regime, where \(\delta_k\downarrow0\) fast enough to recover the minimizers of the original nonsmoothed problem. This should be distinguished from the fixed-parameter ZOPPA framework of \cite{NaldiLabarriereMolinariVilla2026}, in which one analyzes the iteration of \(m_\delta\) for a fixed \(\delta>0\) as an algorithm for a smoothed auxiliary objective.

\subsection{Monotonicity Properties of Barycenter Approximation}\label{monotonicity-sec}
 We begin by recording basic operator-theoretic properties of the map
\(x\mapsto m_\delta(x;f)\). These properties are useful in their own right and also clarify why the barycentric approximation behaves as a regularized surrogate of the proximal mapping. In particular, smoothness and cocoercivity will play an important role in the stability analysis of the iterative methods considered below.

\begin{proposition}\label{nonexpansive}
Let $f\colon \mathbb{R}^n \to \mathbb{R}\cup \{+\infty\}$ be a proper, lower semicontinuous, and $\rho$-weakly convex function for some $\rho\geq 0$. Let $\delta>0$ and assume that $\lambda \in (0,1/\rho)$. Then:
\begin{enumerate}
    \item[(i)] The map $x\mapsto m_{\delta}(x;f)$ is of class $C^{\infty}$ on $\mathbb{R}^n$ and 
    $$Dm_{\delta}(x;f)=\frac{1}{\delta \lambda} \operatorname{Cov}_{\sigma_{\delta}}(Y)\preceq
\frac{\delta\lambda}{1-\lambda\rho} I.$$
    \item[(ii)] The map $x\mapsto m_{\delta}(x; f)$ is $(1-\lambda\rho)$-cocoercive, i.e., for all $x_1, x_2 \in \mathbb{R}^n$,
$$
\langle m_{\delta}(x_1; f) - m_{\delta}(x_2; f), x_1 - x_2 \rangle \geq (1-\lambda \rho) \Vert m_{\delta}(x_1; f) - m_{\delta}(x_2; f) \Vert^2.
$$
Consequently, $x\mapsto m_{\delta}(x; f)$ is monotone and globally Lipschitz with constant $1/(1-\lambda \rho)$. Additionally, if $f$ is convex, then $x\mapsto m_{\delta}(x; f)$ is firmly nonexpansive.
\end{enumerate}
\end{proposition}
\begin{proof}
$(i)$ Let us define the function
\begin{equation}\label{def_Zdelta}
Z_{\delta}(x) := \int_{\mathbb{R}^n} e^{-\frac{1}{\delta}f(y) - \frac{1}{2\delta\lambda}\Vert y-x\Vert^2} dy.
\end{equation}
 By standard properties of exponential families, \(Z_{\delta}\) is of class \(C^\infty\). We can rewrite the barycenter estimator as a gradient mapping:
\begin{equation}\label{Identitymdelta}
m_\delta(x;f) = x + \delta\lambda \nabla \ln Z_\delta(x).
\end{equation}
Thus, \(m_\delta(\cdot;f)\) is \(C^\infty\). Differentiating this expression with respect to \(x\) gives:
\[
Dm_\delta(x;f) = I + \delta\lambda \nabla^2 \ln Z_\delta(x).
\]
By computing the Hessian of the log-partition function, we obtain exactly the scaled covariance matrix, yielding \(Dm_\delta(x;f) = \frac{1}{\delta\lambda} \operatorname{Cov}_{\sigma_\delta}(Y)\).

(ii) Since \(f\) is \(\rho\)-weakly convex, the potential function \(V(y) := \frac{1}{\delta}f(y) + \frac{1}{2\delta\lambda}\Vert y-x\Vert^2\) is strongly convex with parameter \(\mu = \frac{1}{\delta}(\frac{1}{\lambda} - \rho) = \frac{1-\lambda\rho}{\delta\lambda} > 0\). Applying Lemma~\ref{lem:covariance-strongly-log-concave} to the probability 
measure $d\sigma_\delta(y)
=
\frac{e^{-V(y)}}{\int_{\mathbb R^n}e^{-V^\delta(z)}\,dz}\,dy$,  we obtain $\operatorname{Cov}_{\sigma_\delta}(Y)
\preceq
\frac{\delta\lambda}{1-\lambda\rho} I$.  Using the identity from $(i)$, this implies \(0 \preceq Dm_\delta(x;f) \preceq \frac{1}{1-\lambda\rho}I\). Furthermore, because \(m_\delta(\cdot;f)\) is the gradient of the convex function \(\Phi(x) = \frac{1}{2}\Vert x\Vert^2 + \delta\lambda \ln Z_\delta(x)\), the spectral bound on its symmetric Jacobian guarantees that \(m_\delta(\cdot;f)\) is \((1-\lambda\rho)\)-cocoercive. If \(f\) is convex, then \(\rho=0\), making the operator \(1\)-cocoercive, which is equivalent to being firmly nonexpansive (see, e.g., \cite[Proposition 4.4]{MR3616647}).
\end{proof}

\begin{remark}
It steems from Proposition \ref{nonexpansive} that whenever $f$ is convex and for a fixed $\delta>0$ the map $x\mapsto m_{\delta}(x;f)$ is firmly nonexpansive. Hence, we can study the study of fixed point  hence, we can consider 
\end{remark}

\begin{remark}[Bayesian denoising interpretation]
The preceding proposition admits a Bayesian reading. Regarding
$e^{-f/\delta}$ as a (possibly improper) prior density, the measure
$\sigma_\delta$ is the posterior of $Y$ under the Gaussian observation
model $x\sim\mathcal{N}(Y,\delta\lambda I)$, so that
$m_\delta(x;f)=\mathbb{E}[Y\mid x]$ is the corresponding minimum
mean-square-error (MMSE) estimator, that is, a Bayesian denoiser. In this
language, the identity $m_\delta(x;f)=x+\delta\lambda\nabla\ln Z_\delta(x)$ provided in \eqref{Identitymdelta}, for the function $Z_\delta$ introduced in \eqref{def_Zdelta},  is  the so-called Tweedie's formula, and
$Dm_\delta(x;f)=\frac{1}{\delta\lambda}\operatorname{Cov}_{\sigma_\delta}(Y)$
is the associated posterior-covariance identity (see, e.g. \cite{Efron2011} and the references therein). The
firm nonexpansiveness and cocoercivity established in part~(ii) are
precisely the operator-theoretic properties invoked to guarantee
convergence of plug-and-play and regularization-by-denoising schemes built
on MMSE denoisers \cite{XuSunLiuWohlbergKamilov2020,PritchardParhi2025};
here they are obtained analytically, from the weak convexity of $f$ through
the Brascamp-Lieb inequality, rather than imposed through architectural
constraints or training.
\end{remark}

\subsection{Proximal Point Methods via Barycenter Approximation}

We next consider iterative schemes in which the exact proximal step is replaced by its barycentric approximation. As noted above, the convergence results are consequences of the broader framework of inexact proximal point and inexact forward-backward methods (see \cite{MR3942891}). We nevertheless state them explicitly in the present context in order to highlight the specific role of the stochastic approximation error and to show how the bounds established earlier yield concrete convergence estimates for the corresponding barycentric iterations. Replacing exact proximal steps by sampled surrogates of this kind inside splitting algorithms has recently been developed for the closely related HJ-Prox operator \cite{DiChiWuFung2026}, and, on the imaging side, plug-and-play proximal-gradient schemes driven by MMSE denoisers admit analogous nonasymptotic guarantees \cite{PritchardParhi2025}; the distinctive feature here is that the sharp bound of Theorem~\ref{convergence} makes the resulting approximation error fully explicit, namely $\|e_k\|\le c_\lambda\sqrt{\delta_k}$ with $c_\lambda$ given below.

\begin{theorem}\label{TeoConverAlg}
Let $f\colon \R^n\to\R\cup\{+\infty\}$ be a proper, lower semicontinuous, and convex function with
$\operatorname{int}\dom f\neq\emptyset$. Let $h\colon \mathbb{R}^n \to \mathbb{R}$ be a convex $C^1$ function with $L$-Lipschitz gradient  and assume that $\operatorname{argmin} (f+h)\neq\emptyset$.
Fix $\lambda\in (0,2/L)$ and let $(\delta_k)$ with $\delta_k>0$. Consider the sequence generated by iteration
$$
x_{k+1}=m_{\delta_k}(x_k-\lambda \nabla h(x_k);f), \quad k\geq 0.
$$
Then the following assertions hold:
\begin{itemize}
    \item[(i)]  If $\sum \delta_k^{1/2}<\infty$, then $(x_k)$ converges to some point $x^{\ast}\in \operatorname{argmin}(f+h)$.
\item[(ii)] If, in addition, $f$ is $\gamma$-strongly convex for some $\gamma>0$ and $\delta_k\to 0$,  then \(f+h\) has a unique minimizer \(x^\ast\), the sequence \((x_k)\) converges to \(x^\ast\), and
$$
\Vert x_{k}-x^{\ast}\Vert \leq q^k\Vert x_0-x^{\ast}\Vert +\sum_{i=0}^{k-1}q^{k-1-i}\varepsilon_i,
$$
where $q:=1/(1+\lambda \gamma)<1$ and $\varepsilon_k:=\sqrt{\frac{n\delta_k}{\gamma+1/\lambda}}$.
\end{itemize}
\end{theorem}

\begin{proof}
Define $T(x):=\operatorname{prox}_{\lambda f}(x-\lambda \nabla h(x))$. Then \(\operatorname{Fix}T=\argmin(f+h)\). Moreover, since \(h\) is convex and \(\nabla h\) is \(L\)-Lipschitz, the Baillon-Haddad theorem \cite[Corollary 18.16]{MR3616647} implies that \(\nabla h\) is \(1/L\)-cocoercive. Hence, for \(0<\lambda<2/L\), the map \(I-\lambda \nabla h\) is averaged, and therefore nonexpansive.

Since \(\operatorname{prox}_{\lambda f}\) is firmly nonexpansive, it follows that \(T\) is averaged, in particular nonexpansive. Next, by Theorem \ref{convergence} with \(\rho=0\), for every \(z\in\R^n\),
\[
\|m_{\delta_k}(z;f)-\operatorname{prox}_{\lambda f}(z)\|
\le
\sqrt{2\lambda\delta_k}\,
\frac{\Gamma\!\left(\frac{n+1}{2}\right)}
{\Gamma\!\left(\frac{n}{2}\right)}.
\]
Therefore, if we set $e_k:=m_{\delta_k}(x_k-\lambda\nabla h(x_k);f)-T(x_k)$,  then
\[
x_{k+1}=T(x_k)+e_k
\qquad\text{and}\qquad
\|e_k\|
\le
c_\lambda \sqrt{\delta_k},
\]
where $c_\lambda:=\sqrt{2\lambda}\, \Gamma\!\left(\frac{n+1}{2}\right) / \Gamma\!\left(\frac{n}{2}\right)$.\\
$(i)$:  Since \(T\) is averaged, \(\operatorname{Fix}T\neq\emptyset\), and
\(\sum_k \|e_k\|<\infty\), the standard convergence theorem for inexact
averaged iterations (see \cite{MR3942891}) yields that \((x_k)\) converges to some point
\(x^\ast\in \operatorname{Fix}T=\argmin(f+h)\).\\
$(ii)$: If \(f\) is \(\gamma\)-strongly convex, then
\(\operatorname{prox}_{\lambda f}\) is \((1+\lambda\gamma)^{-1}\)-Lipschitz.
 Hence, \(T\) is a contraction with
constant $q:=1/(1+\lambda\gamma)\in(0,1)$. Thus, if \(x^\ast\) denotes the unique fixed point of \(T\), then
\[
\|x_{k+1}-x^\ast\|
\le
q\|x_k-x^\ast\|+\|e_k\|.
\]
Moreover, for each \(k\), the function
\[
y\mapsto f(y)+\frac{1}{2\lambda}\|y-(x_k-\lambda\nabla h(x_k))\|^2
\]
is \((\gamma+1/\lambda)\)-strongly convex. Hence Theorem \ref{convergence}
gives
\[
\|e_k\|
\le
\sqrt{\frac{2\delta_k}{\gamma+1/\lambda}}\,
\frac{\Gamma\!\left(\frac{n+1}{2}\right)}
{\Gamma\!\left(\frac{n}{2}\right)}.
\]
Using the estimate
\[
\frac{\Gamma\!\left(\frac{n+1}{2}\right)}
{\Gamma\!\left(\frac{n}{2}\right)}
\le
\sqrt{\frac{n}{2}},
\]
we obtain
\[
\|e_k\|\le \varepsilon_k,
\qquad
\varepsilon_k:=\sqrt{\frac{n\delta_k}{\gamma+1/\lambda}}.
\]
Therefore,
\[
\|x_{k+1}-x^\ast\|
\le
q\|x_k-x^\ast\|+\varepsilon_k.
\]
Unrolling this recursion yields
\[
\|x_k-x^\ast\|
\le
q^k\|x_0-x^\ast\|
+
\sum_{i=0}^{k-1} q^{\,k-1-i}\varepsilon_i,
\qquad k\ge1.
\]
Since \(\delta_k\to0\), one has \(\varepsilon_k\to0\), and the right-hand
side converges to \(0\). Thus \(x_k\to x^\ast\).
\end{proof}
To make the convergence rate in Theorem~\ref{TeoConverAlg}~(ii) fully explicit when the smoothing parameter $\delta_k$ decays at some rate, we state the following  lemma on sequence convergence, which follows standar thecniques, so we omit its proof.
\begin{lemma}\label{LemmaRates}
Let \(q\in(0,1)\), and let \((\delta_k)_{k\geq 0}\) be a sequence of
nonnegative real numbers. For \(k\geq 1\), define
\[
    S_k:=\sum_{i=0}^{k-1} q^{k-1-i}\sqrt{\delta_i}.
\]
Then the following estimates hold:
\begin{enumerate}
    \item[(i)] If \(\delta_k=\mathcal{O}(k^{-r})\) as \(k\to\infty\), for some
    \(r>0\), then $  S_k=\mathcal{O}(k^{-r/2})$ as $k\to\infty$.

    \item[(ii)] If \(\delta_k=\mathcal{O}(p^k)\) as \(k\to\infty\), for some
    \(p\in(0,1)\), then $  S_k=
        \mathcal{O}\left(
        \max\left\{p^{k/4},q^{k/2}\right\}
        \right)$ as $k\to\infty$.
\end{enumerate}
\end{lemma}
We are now in a position to establish the final result of the paper, concerning the convergence rates of the iterates under strong convexity. In particular, the corollary shows that, when the objective function is strongly convex, the convergence rate of the iterates inherits the convergence rate of the sequence $(\delta_k)$. The proof of this result is a simply combination of Theorem~\ref{TeoConverAlg} and Lemma~\ref{LemmaRates}, so we omit the details.

\begin{corollary}
Under the assumptions of Theorem~\ref{TeoConverAlg}, suppose that \(f\) is
\(\gamma\)-strongly convex for some \(\gamma>0\). Let \(x^\ast\) denote the
unique minimizer of \(f+h\). Then the following convergence rates hold:
\begin{enumerate}
    \item[(i)] If \(\delta_k=\mathcal{O}(k^{-r})\) for some \(r>2\), then $\|x_k-x^\ast\| = \mathcal{O}(p^{k/2})$.

    \item[(ii)] If \(\delta_k=\mathcal{O}(p^k)\) for some \(p\in(0,1)\), then $ \|x_k-x^\ast\|
        =
        \mathcal{O}\left(
        \max\left\{p^{k/4},q^{k/2}\right\}
        \right)$.
\end{enumerate}
\end{corollary}

As a final comment of this section, one may also consider the fixed-parameter regime
\(\delta_k\equiv\delta>0\). In this case, the scheme is no longer an
inexact proximal method with a vanishing error; rather, it naturally leads
to fixed-point schemes, such as the Krasnosel'skii--Mann iteration, applied
to the smoothed barycentric operator \(m_\delta(\cdot;f)\). This viewpoint
is consistent with the SoftMax-function interpretation developed in
\cite{NaldiLabarriereMolinariVilla2026}. Moreover, the existence of fixed
points follows directly from the integrability assumption. Indeed, if
\(e^{-f/\delta}\in L^1(\mathbb R^n)\) and
\(\int_{\mathbb R^n} e^{-f(y)/\delta}\,dy>0\), then the Gaussian
normalization factor associated with \(m_\delta(\cdot;f)\) is positive,
smooth, and vanishes at infinity; hence it attains a maximum. At any
maximizer \(x_\delta\), the first-order optimality condition implies
\(m_\delta(x_\delta;f)=x_\delta\). Thus, under the convex assumptions that
ensure firm nonexpansiveness of \(m_\delta(\cdot;f)\) (see Proposition \ref{nonexpansive}), the corresponding fixed-point iteration converges to a fixed point of the SoftMax-function (see \cite{NaldiLabarriereMolinariVilla2026}). This point is generally not an exact minimizer of the
original nonsmoothed problem, but the estimates obtained in this paper quantify the
deterministic bias introduced by keeping \(\delta>0\): one smoothed
proximal step remains within order \(\sqrt{\delta}\) of the corresponding
classical proximal step, with the explicit constant given in
Theorem~\ref{convergence}.

\section{Numerical Experiments}\label{numerics}

In this section we numerically illustrate the main conclusions of the paper:
the uniform $\sqrt{\delta}$ convergence bound of Theorem~\ref{convergence},
its explicit dimension-dependent constant, and the asymptotic sharpness
mechanism exhibited by the cone $C_\alpha$ in Theorem~\ref{Stochastic-dominance}. We also test the abstract radial dominance bound on a genuinely weakly convex, nonsmooth example. All experiments are fully reproducible; Monte Carlo estimates use $3\times 10^{6}$ samples with fixed seed, and the reported quantities agree with the closed-form expressions to the displayed precision.

\subsection{Convergence rate for projection onto a corner}
Let $C=\mathbb{R}^n_+$ be the nonnegative orthant and take $x=0$, the corner vertex, which is the extremal point of nonsmoothness of the boundary. For the projection estimator ($f=\iota_C$, $\lambda=1$) the deterministic bias admits the closed form $\|p_\delta(0)-P_C(0)\|=\sqrt{\frac{2n\delta}{\pi}}$, so the exact error decays at the rate $\sqrt{\delta}$ predicted by the corollary to Theorem~\ref{convergence}. Figure~\ref{fig:rate} plots, for $n=5$, the theoretical bound $\sqrt{2\delta/\mu}\,\Gamma(\tfrac{n+1}{2})/\Gamma(\tfrac n2)$ (with $\mu=1$), the exact error, and a Monte Carlo estimate, on a log-log scale over $\delta\in[2^{-7},1]$. The three curves are parallel straight lines of slope $\tfrac12$; a linear fit to the Monte Carlo data yields an empirical slope of $0.5000$. The bound is uniformly valid and tight up to the constant factor $\sqrt{n/\pi}\,\Gamma(\tfrac n2)/\Gamma(\tfrac{n+1}{2})$, which equals $0.8385$ for $n=5$ and decreases monotonically to $\sqrt{2/\pi}\approx 0.7979$ as $n\to\infty$.

\begin{figure}[ht]
\centering
\begin{tikzpicture}
\begin{loglogaxis}[
    width=0.72\textwidth, height=0.5\textwidth,
    xlabel={$\delta$}, ylabel={mean displacement},
    legend pos=north west, legend cell align={left},
    grid=both, grid style={gray!20},
    log basis x=2,
]
\addplot[blue, thick, mark=*] coordinates {
(1,2.127692)(0.5,1.504506)(0.25,1.063846)(0.125,0.752253)(0.0625,0.531923)(0.03125,0.376126)(0.015625,0.265962)(0.0078125,0.188063)
};
\addplot[red, thick, mark=square*] coordinates {
(1,1.784124)(0.5,1.261566)(0.25,0.892062)(0.125,0.630783)(0.0625,0.446031)(0.03125,0.315392)(0.015625,0.223016)(0.0078125,0.157696)
};
\addplot[black, only marks, mark=x, mark size=3pt] coordinates {
(1,1.784628)(0.5,1.261822)(0.25,0.891967)(0.125,0.631130)(0.0625,0.445952)(0.03125,0.315488)(0.015625,0.223009)(0.0078125,0.157679)
};
\legend{{Theoretical bound}, {Exact error}, {Monte Carlo}}
\end{loglogaxis}
\end{tikzpicture}
\caption{Convergence of the projection estimator onto the corner of $\mathbb{R}^5_+$ at $x=0$. The theoretical bound of Theorem~\ref{convergence}, the exact mean displacement $\sqrt{2n\delta/\pi}$, and a Monte Carlo estimate all decay at the sharp rate $\sqrt\delta$ (slope $\tfrac12$ on the log-log scale).}
\label{fig:rate}
\end{figure}
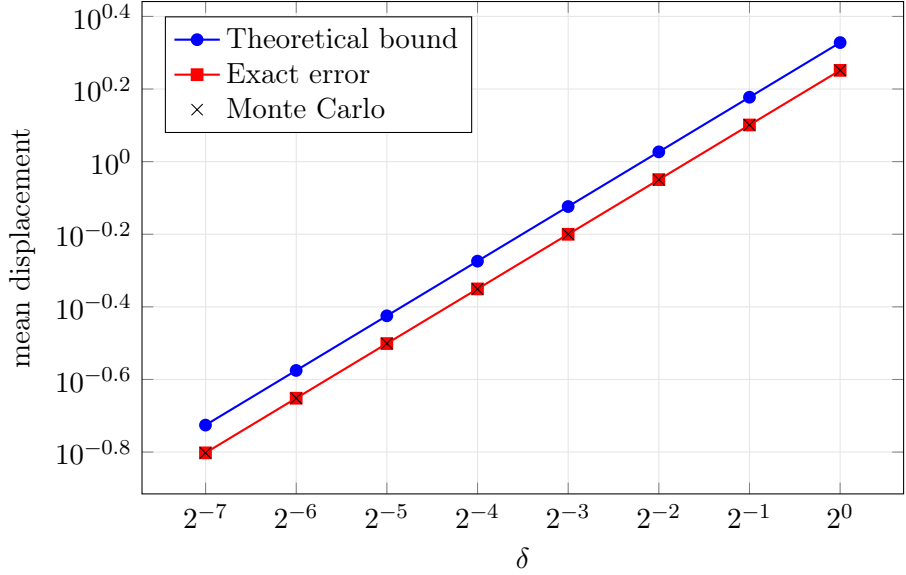

\subsection{Asymptotic sharpness on the spherical-cap cone}
Theorem~\ref{Stochastic-dominance} identifies the spherical cap cone $C_\alpha=\{y:\langle y,e\rangle\ge\|y\|\cos\alpha\}$ of half-angle $\alpha$ as the configuration on which the radial dominance bound becomes asymptotically sharp as $\alpha\to0$. The relevant tightness ratio is $\mathbb{E}[\cos\Theta\mid\Theta\le\alpha]$, where $\Theta$ is the polar angle of a Gaussian vector; the bound is attained in the limit precisely when this ratio tends to $1$. Figure~\ref{fig:sharpness} reports this ratio, computed by numerical quadrature and independently validated by Monte Carlo, as a function of $\alpha$ for $n\in\{2,5,20\}$. In every dimension the ratio increases to $1$ as the half-angle shrinks, confirming that the bound cannot be improved within the radial dominance framework; the convergence is slower in higher dimensions, consistent with the concentration of Gaussian mass away from the axis.
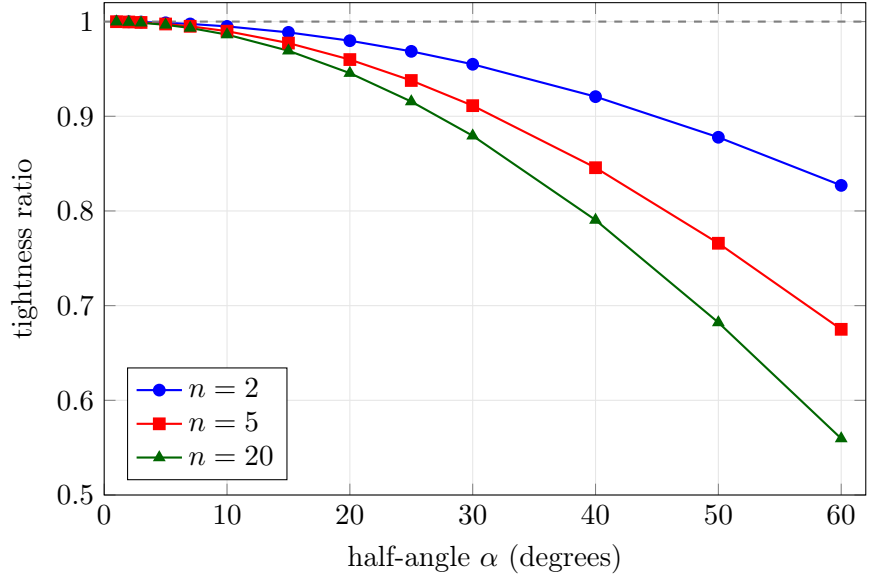
\begin{figure}[htbp]
\centering
\begin{tikzpicture}
\begin{axis}[
    width=0.72\textwidth, height=0.5\textwidth,
    xlabel={half-angle $\alpha$ (degrees)}, ylabel={tightness ratio},
    legend pos=south west, legend cell align={left},
    grid=both, grid style={gray!20},
    xmin=0, xmax=62, ymin=0.5, ymax=1.02,
]
\addplot[blue, thick, mark=*] coordinates {
(60,0.82699)(50,0.87782)(40,0.92073)(30,0.95493)(25,0.96857)(20,0.97982)(15,0.98862)(10,0.99493)(7,0.99751)(5,0.99873)(3,0.99954)(2,0.99980)(1,0.99995)
};
\addplot[red, thick, mark=square*] coordinates {
(60,0.67500)(50,0.76588)(40,0.84568)(30,0.91121)(25,0.93779)(20,0.95990)(15,0.97732)(10,0.98988)(7,0.99503)(5,0.99746)(3,0.99909)(2,0.99959)(1,0.99990)
};
\addplot[black!60!green, thick, mark=triangle*] coordinates {
(60,0.55962)(50,0.68218)(40,0.79040)(30,0.87941)(25,0.91553)(20,0.94556)(15,0.96921)(10,0.98626)(7,0.99326)(5,0.99656)(3,0.99876)(2,0.99945)(1,0.99986)
};
\addplot[gray, dashed, thick, domain=0:62] {1};
\legend{{$n=2$}, {$n=5$}, {$n=20$}}
\end{axis}
\end{tikzpicture}
\caption{Tightness ratio $\mathbb{E}[\cos\Theta\mid\Theta\le\alpha]$ on the cone $C_\alpha$ as the half-angle $\alpha\to0$, for $n\in\{2,5,20\}$. The ratio approaches $1$ in every dimension, confirming the asymptotic sharpness of Theorem~\ref{Stochastic-dominance}; higher dimensions converge more slowly.}
\label{fig:sharpness}
\end{figure}

\subsection{A weakly convex nonsmooth example}
To probe the bound beyond the convex projection setting, we consider the separable weakly convex function $f(y)=\sum_{i=1}^n\big(\tau|y_i|-\tfrac{\rho}{2}y_i^2\big)$ on $\mathbb{R}^n$ with $n=8$, $\lambda=0.5$, $\tau=1$. The parameter $\rho\ge0$ controls the concave (weakly convex) part: $\rho=0$ recovers the convex $\ell_1$ proximal map, while $\rho=0.8$ makes $f$ genuinely $\rho$-weakly convex. Table~\ref{tab:weakly} reports, for two base points and decreasing $\delta$, the measured mean error, the theoretical bound of Theorem~\ref{convergence}, and their ratio. At a generic base point the estimator decays strictly faster than $\sqrt\delta$ because the proximal map is locally smooth there, so the ratio drifts toward $0$; at the kink configuration $x_i=\lambda\tau$, where the proximal map is nonsmooth, the error decays at exactly the $\sqrt\delta$ rate and the ratio saturates near $0.8$, matching the sharp constant. In all cases the bound is respected uniformly in $\delta$.

\begin{table}[htbp]
\centering
\caption{Weakly convex separable example ($n=8$, $\lambda=0.5$, $\tau=1$). Mean error, theoretical bound, and their ratio for two base points and decreasing $\delta$. At the nonsmooth kink configuration the ratio saturates near the sharp constant; at a generic point the error decays faster than $\sqrt\delta$.}
\label{tab:weakly}
\begin{tabular}{@{}llccc@{}}
\toprule
 {Configuration} &  {$\delta$} &  {Mean error} &  {Bound} &  {Ratio} \\
\midrule
 {Generic $x$, $\rho=0$}        &  {$0.5$}      &  {$0.3222$} &  {$1.3708$} &  {$0.235$} \\
 {}                             &  {$0.125$}    &  {$0.1508$} &  {$0.6854$} &  {$0.220$} \\
 {}                             &  {$0.03125$}  &  {$0.0483$} &  {$0.3427$} &  {$0.141$} \\
 {}                             &  {$0.0078125$}&  {$0.0106$} &  {$0.1714$} &  {$0.062$} \\
\midrule
 {Generic $x$, $\rho=0.8$}      &  {$0.5$}      &  {$0.4174$} &  {$1.7697$} &  {$0.236$} \\
 {}                             &  {$0.125$}    &  {$0.1752$} &  {$0.8849$} &  {$0.198$} \\
 {}                             &  {$0.03125$}  &  {$0.0463$} &  {$0.4424$} &  {$0.105$} \\
 {}                             &  {$0.0078125$}&  {$0.0108$} &  {$0.2212$} &  {$0.049$} \\
\midrule
 {Kink $x_i=\lambda\tau$, $\rho=0.8$} &  {$0.5$}      &  {$1.0178$} &  {$1.7697$} &  {$0.575$} \\
 {}                             &  {$0.125$}    &  {$0.6123$} &  {$0.8849$} &  {$0.692$} \\
 {}                             &  {$0.03125$}  &  {$0.3352$} &  {$0.4424$} &  {$0.758$} \\
 {}                             &  {$0.0078125$}&  {$0.1749$} &  {$0.2212$} &  {$0.791$} \\
\bottomrule
\end{tabular}
\end{table}

\section{Conclusions and Future Perspectives}

In this paper, we introduced a radial dominance framework to study the approximation properties of stochastic barycentric estimators. By bounding the estimator's performance against an explicit one-dimensional comparison measure, we established an asymptotically sharp, nonasymptotic convergence rate for proximal operators of weakly convex functions and metric projections onto closed convex sets. Furthermore, we demonstrated that these barycentric operators inherit desirable structural properties, such as smoothness and cocoercivity, making them highly suitable for proximal-type algorithms. The theoretical predictions were corroborated numerically in Section~\ref{numerics}: the projection estimator on the orthant attains the predicted $\sqrt\delta$ rate and remains uniformly below the sharp dimension-dependent bound, the tightness ratio on the cone $C_\alpha$ confirms the asymptotic sharpness of Theorem~\ref{Stochastic-dominance}, and the bound remains valid on a genuinely weakly convex nonsmooth example, approaching the corresponding nonsmooth worst-case ratio at kink configurations.

Several interesting avenues for future research remain open. First, while our analysis relies heavily on the rotational symmetry of the Gaussian perturbations, it would be of interest to extend this radial dominance framework to heavy-tailed distributions (such as Student's $t$ or Cauchy), which are often used in stochastic search to encourage exploration. Second, extending these explicit dimensional bounds to infinite-dimensional Hilbert spaces or specific Banach spaces remains a challenging open question.

\section*{Acknowledgements}

Pedro P\'erez-Aros was supported by ANID (Chile) through Fondecyt Regular grants No.~1240120 and No.~1261728, CMM BASAL funding for the Center of Excellence FB210005; and the project ECOS230027. 

Emilio Vilches was supported by ANID (Chile) through Fondecyt Regular grants No.~1240120 and No.~1261728, CMM BASAL funding for the Center of Excellence FB210005, and the project ECOS230027.

\bibliographystyle{abbrv}
\bibliography{bib.bib}




\end{document}